\documentclass[leqno,11pt, a4]{amsart}
\usepackage[usenames,dvipsnames,svgnames,table]{xcolor}

\usepackage[active]{srcltx}
\usepackage{pb-diagram}

\usepackage{dsfont}

\usepackage{mathrsfs}
\usepackage{amsmath}
\usepackage{amssymb}
\usepackage{amscd}
\usepackage{amsthm}
\usepackage[latin1]{inputenc}
\usepackage{graphics}
\usepackage{varioref}

\usepackage[latin1]{inputenc}
\usepackage{graphics}

\labelformat{enumi}{(#1)}

\newtheorem{teo}[equation]{Theorem}
\newtheorem{defin}[equation]{Definition}
\newtheorem{remark}[equation]{Remark}
\newtheorem{prop}[equation]{Proposition}
\newtheorem{cor}[equation]{Corollary}
\newtheorem{lemma}[equation]{Lemma}
\newtheorem{Example}[equation]{Example}

\newtheoremstyle{named}{}{}{\itshape}{}{\bfseries}{.}{.5em}{\thmnote{#3}#1}
\theoremstyle{named}

\newcommand{\ga}{\gamma}
\newcommand{\dga}{\dot{\gamma}}

\newcommand{\meno}{^{-1}}

\newcommand{\La}{\Lambda}

\newcommand{\proba}{\mathscr{P}}
\newcommand{\pb}[1]{\proba(#1)}
\newcommand{\misu}{\mathscr{M}}

\newcommand{\XX}{\mathfrak{X}}

\newcommand{\psiM}{\psi^\proba}
\newcommand{\PsiM}{\Psi^\proba}

\newcommand{\liu}{\mathfrak{u}}
\newcommand{\liek}{\mathfrak{k}}
\newcommand{\liel}{\mathfrak{l}}
\newcommand{\lieg}{\mathfrak{g}}

\newcommand{\liep}{\mathfrak{p}}
\newcommand{\lieq}{\mathfrak{q}}

\newcommand{\lia}{\mathfrak{a}}

\newcommand{\grad}{\operatorname{grad}}

\newcommand{\alfa}{\alpha}
\newcommand{\alf}{\alpha}
\newcommand{\vacuo}{\emptyset}
\newcommand{\im}{\operatorname{\pi_\liep}}
 \newcommand{\pf}{{}_*}
\newcommand{\est}{\Lambda}

\newcommand{\la}{\lambda}

\newcommand{\enf}{\emph}

\newcommand{\desudt}[1] []      {\dfrac {\mathrm {d} #1 }{\mathrm {dt}}}
\newcommand{\desudtzero}        {\desudt \bigg \vert _{t=0} }
\newcommand{\deze}        {\desudt \bigg \vert _{t=0} }

\newcommand{\restr}[1]          {\vert_{#1}}

\newcommand{\Ad}{\operatorname{Ad}}

\newcommand{\sx}{\langle}
\newcommand{\xs}{\rangle}
\newcommand{\scalo}{\sx \cdot , \cdot \xs}

\newcommand{\cds}{\cdots}
\newcommand{\cd}{\cdot}
\renewcommand{\setminus}{-}
\newcommand{\cinf}{C^\infty}

\newcommand{\ra}{\rightarrow}
\newcommand{\lra}{\longrightarrow}
\newcommand{\C}{\mathbb{C}}
\newcommand{\R}{\mathbb{R}}

\renewcommand{\phi}{\varphi}
\renewcommand{\bigl}{\left}
\renewcommand{\biggl}{\left}
\renewcommand{\Bigl}{\left}
\renewcommand{\bigr}{\right}
\renewcommand{\biggr}{\right}
\renewcommand{\Bigr}{\right}

\newcommand{\fun}{\mathfrak{F}}

\newcommand{\Crit}{\operatorname{Crit}}

\newcommand{\Id}{\operatorname{id}}

\newcommand{\lieh}{\mathfrak{h}}
\newcommand{\e}{\operatorname{e}}

\newcommand{\spaz}{\mathscr{M}}
\newcommand{\mume}{\fun\meno(0)}
\newcommand{\unfi}{u_\infty}
\newcommand{\tits}{\partial_\infty X}
\newcommand{\limes}{\alfa}

\newcommand{\x}{{v}}

\newcommand{\weak}{\rightharpoonup}

\newcommand{\meo}{\end{document}}
\newcommand{\mup}{\mu_\liep}
\newcommand{\mua}{\mu_\lia}
\newcommand{\mupb}{\mu_\liep^\beta}
 
\begin{document}
\title{Stability with respect to actions of real reductive Lie groups}
\author{Leonardo Biliotti}
\address{(Leonardo Biliotti) Dipartimento di Matematica e Informatica \\
          Universit\`a di Parma (Italy)}
\email{leonardo.biliotti@unipr.it}

\author{Michela Zedda}
\address{(Michela Zedda) Dipartimento di Matematica e Fisica ``Ennio De Giorgi'' \\
          Universit\`a del Salento (Italy)}
\email{michela.zedda@gmail.com}

\begin{abstract}
We give a systematic treatment of the stability theory for action of a real reductive Lie group $G$ on a topological space. More precisely, we introduce an abstract setting for actions of non-compact real reductive Lie groups on topological spaces that admit functions
similar to the  Kempf-Ness function. The point of this construction
is that one can characterize stability, semi-stability and polystability of a point  by numerical criteria, that is in terms of a function called maximal weight. We apply this setting to the actions of a real non-compact reductive Lie group $G$ on a real compact submanifold $M$ of a K\"ahler manifold $Z$ and to the action of $G$ on measures of $M$.
\end{abstract}

%
\thanks{The authors were partially supported by FIRB 2012 ``Geometria
  differenziale e teoria geometrica delle funzioni''. Further, the authors were also supported by GNSAGA of
  INdAM and by MIUR PRIN  2010-2011
   ``Variet\`a reali e complesse: geometria, topologia e analisi armonica''. }

\maketitle

\tableofcontents
\section{Introduction}
\label{sec:introduction}
Stability theory in K\"ahler geometry has been intensively studied by many authors and from several points of view, see e.g. \cite{georgula,heinzner-GIT-stein,heinzner-huckleberry-Inventiones,heinzner-huckleberry-loose,heinzner-loose,kempf-ness,kirwan,mumford-GIT,sjamaar-Annals,schwartz}.
This paper is inspired by the works of I. Mundet i Riera \cite{mundet-Trans} and A. Teleman \cite{teleman-symplectic-stability} where a systematical presentation of the stability theory in the non-algebraic K\"ahlerian geometry of complex reductive Lie groups is given, and by the recent paper \cite{bgs} where the first author jointly with A. Ghigi
develops a geometrical invariant theory on topological spaces, without assuming the existence of a symplectic structure. In particular, they apply the main results to the action of $U^\C$ on measures on a compact K\"ahler manifold $Z$, where $U$ is a compact connected Lie group acting in Hamiltonian fashion on $Z$. This was also motivated by an application to upper bounds for the first eigenvalue of the Laplacian on functions \cite{arezzo-ghigi-loi,biliotti-ghigi-AIF,biliotti-ghigi-American,bourguignon-li-yau,hersch}.


In this paper we identify an  abstract setting to develop the  geometrical invariant theory for actions of real reductive Lie groups.
More precisely, given a Hausdorff topological space $\spaz$ with a continuous action of a non-compact real reductive  Lie group $G=K\exp (\liep)$ and a set of functions formally similar to the classical Kempf-Ness
function  we define an analogue of the gradient map $\fun: \spaz \lra \liep$ and the usual concepts of stability.

The gradient map has been intensively studied in \cite{heinz-stoezel,heinzner-schwarz-Cartan,heinzner-schwarz-stoetzel,heinzner-schutzdeller} and many other papers. The main idea is to investigate
a class of actions of real reductive Lie groups on complex spaces and on real submanifolds using momentum map techniques. This means that we consider a K\"ahler manifold $(Z,\omega)$  acted on by a complex reductive
Lie group $U^\C$ of holomorphic maps. The K\"ahler form $\omega$ is $U$-invariant, where $U$ is a compact form of $U^\C$, and there exists a momentum map $\mu:Z \lra \liu^*$. We recall that a momentum map $\mu$ is $U$-equivariant and for any $\xi\in  \liu$, the gradient of the
function $\mu^\xi (x)=\mu(x)(\xi)$ is given by $J( \xi_Z )$, where $\xi_Z (p)=\frac{d}{dt}\vert_{t=0} \exp(t\xi)p$ is the vector field corresponding to $\xi \in \liu$ and $J$ is the complex structure of $Z$ (see \cite{heinzner-huckleberry-MSRI,mcduff-salamon-symplectic} for more details about momentum map).
Since $U$ is compact we may identify $\liu \cong \liu^*$ by means of an $\mathrm{Ad}(U)$-invariant scalar product on $\liu$. Hence we may think the momentum map as a $\liu$-valued map, i.e., $\mu:Z \lra \liu$.

Let $G \subset U^\C$ be compatible. Then $G$ is closed and the Cartan decomposition $U^\C =U\exp (i \liu)$ induces a Cartan decomposition $G=K\exp (\liep)$, where $K=G\cap U$ and $\liep=\lieg\cap i \liu$.
Identifying $i\liu \cong \liu$ the inclusion $\liep \hookrightarrow i\liu$ induces a $K$-equivariant map $\mup:Z \lra \liep$. Finally if $M$ is a $G$-stable real submanifold of $Z$, we may restrict $\mup$ to $M$ and so considering $\mup:M \lra \liep$. The map $\mup:M \lra \liep$ is called \emph{gradient map}. In Section \ref{gradient-map} we extend the construction given in \cite{mundet-Crelles}  for the gradient map, defining a Kempf-Ness function of $(M,G,K)$.

The $G$-action on $M$ induces in a natural way a continuous action on measures of $M$, that we denote by $\mathcal P (M)$, with respect to the weak-$\ast$ topology. In Section \ref{sez-misure} we prove there exists a Kempf-Ness function for $(\mathcal P (M),G,K)$ and the map
\[
\fun (\nu)=\int_M \mup(x) \mathrm d \nu (x),
\]
is the analogue of the gradient map in this setting.
These are our  basic examples and the main motivations  to develop a geometrical invariant theory for actions of real reductive Lie groups.

Stability and semi-stability are checked using the position of the  $G$-orbit with respect to the vanishing locus of the gradient map.
The main point of our construction is that one can characterize stability, semi-stability and polystability of a point by numerical criteria, that is in terms of a function called maximal
weight, which is defined on the Tits boundary of the symmetric space of non-compact type $G/K$.  Roughly speaking we extend criteria for stability, semi-stability and polystability due to Teleman \cite{teleman-symplectic-stability}, Mundet I Riera \cite{mundet-Crelles, mundet-Trans}, Kapovich, Leeb and Milson \cite{kapovich-leeb-millson-convex-JDG}, Biliotti and Ghigi \cite{bgs} and probably many others, for a large class of  actions of complex reductive Lie groups,  to actions of non-compact real reductive Lie groups. Our  criterion for polystability is weaker than those proved by Mundet i Riera \cite{mundet-Trans} and by the first author and Ghigi in \cite{bgs} for complex reductive Lie gropus. However if $G=K^\C=K\exp(\mathbf{i} \liek)$ is complex reductive then condition $(P3)$ in Section \ref{sec:abstract-setting}, i.e.,   $\frac{\mathrm{d^2}}{\mathrm{dt}^2 }\bigg \vert_{t=0}\Psi(x,\exp(t\x)) =0$
if and only if $\exp(\R \x) \subset G_x$, does not imply $\exp (\C \x) \subset G_x$ as required by $(P3)$ in \cite[p. 6]{bgs}. 
 This condition is crucial in Mundet's proof \cite{mundet-Trans} and in the proof given in \cite{bgs} for polystability. Indeed, thanks to the $K$-equivariance of  $\fun$, if $\exp (\C \x) \subset G_x$, then  $\fun (x) \in \liek^{\x}=\{u\in \liek:\, [u,\x]=0\}$  and thus a sort of a reduction principle applied.

In the abstract setting introduced in this paper, the above condition is equivalent to the following: if $\exp(\R \x)\subset G_x$ then $\fun (x)\in \liep^{\x}=\{u\in \liep:\, [u,\x]=0\}$. This does not hold for a  general gradient map $\fun$ since it is only $K$-equivariant. On the other hand this condition holds for the gradient map \cite{heinzner-schwarz-stoetzel} and the gradient map defined by the Kempf-Ness function with respect to the $G$ action on measure (Proposition \ref{P5M}). The authors believe that the polystability criterion due to Mundet \cite{mundet-Trans} holds under the above condition. We leave this problem for future investigation.

What is satisfactory of  Theorem  \ref{polystable-condition} is that the reductivity of the stabilizer is obtained as a consequence of conditions  involving only the maximal weight and the set on which the maximal weight is zero.
We also prove a version of the Hilbert-Munford criterion and the openness of the set of stable points.
Finally we completely characterize stable, semi-stable and polystable measures on  real projective spaces.

The paper is organized as follows.

In Section \ref{titssection} we review basic facts on real reductive Lie groups and Tits boundary of a Hadamard manifold.

In Section \ref{sec:abstract-setting} we define the abstract setting and the general gradient map with respect to a Kempf-Ness function of $(\spaz, G,K)$.

In Section \ref{maximal-weight-section} we define the maximal weight on the Tits boundary of $X=G/K$. Since the Kempf-Ness function is $K$-invariant, for any $x\in \spaz$  the Kempf-Ness function descends to a function $\psi_x :X \lra \R$ which is geodesically convex. If $\psi_x$ is Lipchitz then Lemma \ref{convex-function} defines what is called  maximal weight on the Tits boundary of $X$. We also point out that the maximal weight is $G$-equivariant.

In Section \ref{sec:stab-with-resp} we define  stable, semi-stable and polystable points giving a numerical criterion for an element $x\in \spaz$
to be stable (Theorem \ref{stabile}). We give a version of the Hilbert-Munford criterion (Corollary \ref{Hilbert-Mumford}) and we prove the openness of the set of stable points (Corollary \ref{stable-open-abstract-setting}).

In Section \ref{sec:polyst-semi-stab} we give numerical criteria for semi-stability (Theorem \ref{semi-stable-abstract}) and polystability (Theorem \ref{polystable-condition}) and a Hilbert-Munford criterion for semi-stable points (Corollary \ref{Hilbert-Mumfords}).

In Section \ref{gradient-map} we discuss the basic example, i.e., the classical gradient map, and in Section \ref{sez-misure} we apply our setting on the $G$ action on measures.
Using the Morse-Bott theory of the gradient map on $M$ we compute rather explicitly the maximal weight. Moreover, assuming the convexity theorem holds for abelian subgroup $A=\exp(\lia)$ where $\lia \subset \liep$, see \cite{bghc,heinzner-schutzdeller} for more details, if $0\in E(\mup)$, where $E(\mup)$ is the convex hull of the image of the gradient map $\mup$, then any smooth measures is semi-stable (Proposition \ref{measure-semistable}). The condition $0\in E(\mup)$ is always satisfied up to shifting the gradient map with respect to some $\mathrm{Ad}(K)$-fixed point of $\liep$. We also prove that the set of semi-stable points is dense. If $0$ lies in the interior of $E(\mup)$ then any smooth measures is stable and the set of stable points is open and dense. This condition is always satisfied if $M$ is an adjoint orbit of $K$ in $\liep$ and $K$ acts irreducibly on $\liep$.

In Section \ref{misure-proiettivo} we completely describe stable, semi-stable and polystable measures on real projective spaces.

\noindent
{\bfseries \noindent{Acknowledgements.}}  The authors are very grateful to Alessandro Ghigi for all the intersting discussions and to Alberto Raffero for his comments. The second author wishes also to thank Parma's Mathematics Department for the wonderful hospitality during her stay as post-doc student.

\section{Tits boundary of $G/K$}\label{titssection}
Let $G$ be a non-compact real reductive Lie group and denote by $\lieg$ its Lie algebra. Recall that such $G$ has a finite number of connected components and its algebra splits as $\lieg=[\lieg ,\lieg]\oplus \mathfrak z (\lieg)$, where $[\lieg ,\lieg]$ is semisimple and $\mathfrak z (\lieg)$ is the center of $\lieg$.
Further, maximal compact subgroups of $G$ always exist and meet every connected components, and any two of them are conjugate under an element of the identity component $G^o$ of $G$.
 Assume that there exists a Cartan involution $\theta:G \lra G$ with fixed points set $K$ and let us denote also by $\theta:\lieg \lra \lieg$ its differential. Then $\lieg=\liek\oplus \liep$ and the map $f: K \times \liep \ra G$, $f(g, \x ) = g  \exp \x $ is a
diffeomorphism. This means that $G=K\exp(\liep)$ and $G/K$ is simply connected. Since $\theta_{|_{\liek}}=\mathrm{Id}$ and $\theta_{|_{\liep}}=-\mathrm{Id}$, we have
$[\liek,\liek]\subset \liek$, $[\liek,\liep]\subset \liep$ and $[\liep,\liep]\subset \liek$. Therefore if $\lia \subset \liep$ is a Lie subalgebra, then it must be abelian. Moreover, two maximal abelian subalgebras contained in $\liep$ are conjugate with respect to the identity component $K^o$. We refer the reader to \cite{borel-ji-libro,helgason, knapp-beyond} for more details on real reductive Lie groups.
Set
 \begin{gather*}
    X:=G/K.
  \end{gather*}
  Observe that $G$ acts isometrically on $X$ from the left by:
  \begin{gather*}
    L_g : X \ra X , \qquad L_g(hK) :=ghK,\qquad g\in G.
  \end{gather*}
To simplify the notation, we will often write $gx$ instead of $L_g(x)$.
The choice of an $\mathrm{Ad}(K)$-invariant scalar product on $\liep$ induces a $G$-invariant Riemannian metric on $X$. It is well known that $X$ endowed with this metric is a symmetric space of non-compact type and thus a Hadamard manifold \cite{eberlein, helgason}. The Riemannian exponential map arises  by the exponential map of Lie groups. Hence a geodesic on $X$ is given by $g \exp (tv)K$, where  $g\in G$ and $v\in \liep$. In the sequel we denote by $\gamma^v (t)=\exp(tv) K$.

Since $X$ is a Hadamard manifold there is a natural notion of boundary at infinity $\partial_\infty X$ which can be described using geodesic.

Two unit speed geodesic rays $\ga, \ga' : (0,+\infty) \ra X$
  are equivalent, denoted by $\ga \sim \ga'$, if
  $\sup_{t\in (0,+\infty)} d(\ga(t), \ga'(t)) $ $ < +\infty$.  The \emph{Tits
    boundary} of $X$, denoted by $\tits$, is the set of equivalence
  classes of unit speed geodesic ray in $X$.

  Set $o:= K \in X$.  Mapping $v$ to the tangent
  vector $\dot{\ga}^v(0)$ yields an isomorphism $\liep \cong T_oX$.
  Since any geodesic ray in $X$ is equivalent to a unique ray starting
  from $o$, the map
  \begin{gather}
    \label{def-e}
    \e : S(\liep) \ra \tits, \ \e(v):= [\ga^v],
  \end{gather}
  where $S(\liep)$ is the unit sphere in $\liep$, is a bijection. The
  \emph{sphere topology} is the topology on $\tits$ such that $\e$ is
  a homeomorphism.  (For more details on the Tits boundary see for
  example \cite[\S I.2]{borel-ji-libro} and \cite{eberlein}.)

Since $G$ acts by isometries on $X$, if $\gamma$ is a unit speed geodesic in $X$, then for each $g\in G$ also $g\gamma$ is. Further, since $\gamma \sim \gamma'$ implies $g\gamma \sim g\gamma'$, we get a $G$-action on the Tits boundary $\tits$ by:
\[
g\cdot [\gamma] =[g\gamma],
\]
which also induces by \eqref{def-e} a $G$-action on $S(\liep)$ given by: 
\[
g\cdot v=\e^{-1}(g\cdot \e (v)).
\]\label{ss1}
Observe that this last one is discontinuous with respect to the sphere topology on $S(\liep)$.
\begin{defin}
Let $H\subset G$ be a closed subgroup. Set
  $L:=H\cap K$ and $\tilde \liep:= \lieh \cap \liep$. According to \cite{heinzner-schwarz-stoetzel,
    heinzner-stoetzel-global}, we say that $H$ is
  \enf{compatible} if $H=L \exp (\tilde \liep )$.
\end{defin}
If $H$ is a compatible subgroup of $G$, then it follows that it is a real reductive subgroup of $G$, the Cartan involution of $G$ induces a Cartan involution of $H$, $L$ is a maximal compact subgroup of $H$ and finally $\lieh = \liel \oplus \tilde \liep$. Note that $H$ has finitely many
connected components. Moreover, there are
  totally geodesic inclusions $X':= H/L \hookrightarrow X$ and
  $\partial_\infty X' \subset \tits$.
\section{Kempf-Ness functions}
\label{sec:abstract-setting}
  Let $\spaz$ be a Hausdorff topological space and let $G$ be a non-compact real reductive group which acts continuously on $\spaz$. Observe that with these assumptions we can write $G=K\exp (\liep)$, where $K$ is a maximal compact subgroup of $G$. Starting with these data we  consider a
  function $ \Psi : \spaz \times G \ra \R$, subject to five conditions.
  The first four are the following ones:
  \begin{enumerate}
  \item[($P1$)] For any $x\in \spaz$ the function $ \Psi(x,\cd )$
    is smooth on $G$.
  \item[($P2$)] The function $\Psi(x, \cd )$ is left--invariant
    with respect to $K$, i.e.: $\Psi(x,kg) = \Psi(x,g)$.
  \item[($P3$)] For any $x\in \spaz$, and any $\x \in \liep$ and
    $t\in \R${:}
    \begin{gather*}
      \frac{\mathrm{d^2}}{\mathrm{dt}^2 } \Psi(x,\exp(t\x)) \geq 0.
    \end{gather*}
    Moreover:
    \begin{gather*}
      \frac{\mathrm{d^2}}{\mathrm{dt}^2 }\bigg \vert_{t=0}
      \Psi(x,\exp(t\x)) = 0
    \end{gather*}
    if and only if $\exp(\R \x) \subset G_x$.
  \item[($P4$)] For any $x\in \spaz$, and any $g, h\in G$:
    \begin{gather*}
      \Psi(x,g) + \Psi({gx}, h) = \Psi(x,hg).
    \end{gather*}
    This equation is called the \emph{cocycle condition}.
  \end{enumerate}
  As in the previous section, let $X=G/K.$
If $\Psi$ is a
  function satisfying ($P1$)--($P4$), then by ($P2$) the
  function $g\mapsto \Psi(x, g\meno)$ descends to a function on $X$:
  \begin{gather}
    \label{defpsi}
    \psi_x: X \ra\R, \quad \psi_x(gK) := \Psi(x, g\meno),
  \end{gather}
  and the cocycle condition $(P4)$ can be rewritten in terms of $\psi_x$ as:
  \begin{gather}
    \tag{$P4'$}
    \label{cociclo-psi}
    \psi_x(ghK) = \psi_x (gK) + \psi_{g\meno x} (hK),
  \end{gather}
which is also equivalent to the following identity between
  two functions and a constant:
  \begin{gather}
    \label{cociclo-3}
    L_g^* \psi_x = \psi_{g\meno x} + \psi_x(gK),
  \end{gather}
\label{say-def-la}
where $L_g$ denotes the action of $G$ on $X$ (see previous section).

In order to state our fifth condition, let $\scalo : \liep^*\times \liep \ra \R$ be the duality pairing.
For $x\in \spaz$ define $\fun(x) \in \liep^*$
by requiring that:
\begin{equation}\label{momento-astratto}
  \sx \fun (x), \x\xs = -  (d \psi_x )_{o}(\dga^\x(0))  = \desudtzero \psi_x( \exp(-t\x) K)  = \desudtzero \Psi(x,
    \exp(t\x)) .\nonumber
\end{equation}
The following is the fifth and last condition
  imposed on the function $\Psi$:
  \begin{enumerate}
    \item[$(P5)$]
    The map $\fun : \spaz \ra \liep^*$ is continuous.
  \end{enumerate}

  We call $\fun$ the \emph{gradient map} of $(\spaz, G, K, \Psi).$ As immediate consequence of the definition of $\fun$ we have the following result.
   \begin{prop}\label{equivarianza}
  The map $\fun : \spaz \ra \liep^*$ is $K$-equivariant.
\end{prop}
\begin{proof}
  It is an easy application of the cocycle condition and the
  left-invariance with respect to $K$ of $\Psi(x,\cdot)$. Indeed,
  \begin{gather*}
    \begin{split}
      \sx \fun (kx),\x \xs &=\deze \Psi (x, \exp(t\x)k)=\deze \Psi
      (x,k^{-1} \exp(t\x) k)\\
      &=\deze \Psi\left (x,\exp(t\mathrm{Ad}(k^{-1})(\x))\right
      )=\mathrm{Ad}^* (k) ( \fun (x))(\x) .
    \end{split}
  \end{gather*}
\end{proof}

  The following definition summarizes the above discussion.
\begin{defin}
\label{def-kn}
Let $G$ be a non-compact real reductive Lie group, $K$ a maximal compact subgroup of $G$ and $\spaz$ a
  topological space with a continuous $G$--action. A \emph{Kempf-Ness
    function} for $(\spaz, G,K)$ is a function
  \begin{gather*}
    \Psi : \spaz \times G \ra \R ,
  \end{gather*}
  that satisfies conditions ($P1$)--($P5$).
\end{defin}
\begin{remark}\label{stabilizzatore-lineare}
Taking $g=h=e$ in the cocycle condition ($P4$) we
have $\Psi(x, e) = 0$. Hence $\Psi(x,k)=0$ for every $k\in K$, since $\Psi(x,\cdot)$ is $K$-invariant on the second factor.
Moreover, for any $x\in \spaz$ and for any $g,h\in G_x$ we have:
  \begin{equation}
    \label{somma}
    \Psi(x,hg)=\Psi(x,g)+\Psi(x,h),
  \end{equation}
which implies that $\Psi(x,\cdot):G_x \lra \R$ is a homomorphism.
\end{remark}
\section{Maximal weights}\label{maximal-weight-section}
Let $X=G/K$ and let $u: X \ra \R$ be a smooth function. We say that $u$ is \emph{geodesically convex} on $X$ if  $u(\gamma(t))$ is a convex function for any geodesic $\gamma (t)$ in $X$.
The following lemma is proven in greater generality by Kapovich, Leeb
  and Millson in
  \cite[\S3.1]{kapovich-leeb-millson-convex-JDG} (see also \cite[\S 2.3]{bgs}).
\begin{lemma}\label{convex-function}
  Let $u: X \ra \R$ be a smooth geodesically convex  function on $X$. Assume that
  $u$ is globally Lipschitz continuous.  Then the function
  $\unfi : \tits \ra \R$ {given by:}
  \begin{gather}
    \label{eq:1}
    \unfi ([\ga]) : = \lim _{t \to +\infty } (u\circ\ga)'(t),
  \end{gather}
  is well--defined. Moreover $u$ is an exhaustion if and only if
  $\unfi > 0$ on $\tits$.
\end{lemma}
Recall that a continuous function $f: X \ra \R$
is an exhaustion if for any $c\in \R$ the set
  $f\meno ((-\infty, c])$ is compact, condition which is equivalent for $f$ to be bounded below and proper.

As in \cite{bgs}, the following result holds.
\begin{lemma}
  \label{convstab}
  The function $\psi_x$ is geodesically convex on $X$. More precisely,
  if $\x \in \liep$ and $\alf(t) = g \exp(t\x) K$ is a geodesic in
  $X$, then $\psi_x \circ \alfa$ is either strictly convex or
  affine. The latter case occurs if and only if
  $ g \exp (\R \x) g \meno \subset G_x$.  In the case $g=e$, the
  function $\psi_x \circ \alfa$ is linear if
  $ \exp (\R \x) \subset G_x$ and strictly convex otherwise.
\end{lemma}

{Due to Lemma \ref{convstab},} in order to apply Lemma \ref{convex-function} to $\psi_x$, we need {only} to add this last assumption:

\begin{enumerate}
 \item[$(P6)$]
For any $x\in \spaz$, the function $\psi_x : X \ra \R$ is globally Lipschitz on $X$.
 \end{enumerate}

When property ($P6$) holds, for any $x\in \spaz$ the function $\la_x:=(\psi_x)_\infty$ given by:
  \begin{gather}    \la_x : \tits \ra \R 
    \qquad \la_x ([\ga]): = \lim_{t\to +\infty} \desudt
    \psi_x(\ga(t)),
  \end{gather}
  is well-defined and finite. We call $\la_x$ \emph{maximal weight}.
Moreover for any $x\in \spaz$, any $g\in G$ and any $p\in \tits$ we have (see \cite[Lemma 2.28]{bgs} for a proof):
\begin{gather}\label{maximal-weight-Ginvariant}
\lambda_{g^{-1}x}(p)=\lambda_x (g\cdot p).
\end{gather}
It is also well-defined and finite the function:
  \begin{gather}
    \label{def-la}
    \la : \spaz \times \tits \lra \R , \quad \la(x, p):= \la_x(p).
  \end{gather}
 Since we set the sphere topology on $\tits$, i.e., the topology on $\tits$ such that  $e:S(\liep) \ra \tits$ is an homeomorphism {(see Section \ref{titssection}), by \cite[Lemma 4.9]{bgs}}, $\la$ is lower semicontinuous
and for
  $v\in S(\liep)$ it follows:
    \begin{gather}
      \label{la-exp}
      \la_x (\e(v)) = \lim_{t\to +\infty} \desudt \psi_x(\exp(tv)K) =
      \lim_{t\to +\infty} \desudt \Psi(x, \exp(-tv)).
    \end{gather}
\section{Stability}
 \label{sec:stab-with-resp}
 Let $(\spaz,G, K) $ be as above and let $\Psi$ be a Kempf-Ness
 function. In particular, according to Definition \ref{def-kn} we assume that $\Psi$ satisfies conditions ($P1$)--($P5$).
 \begin{defin}
   \label{stabilita}
   Let $x\in \spaz$. Then:
   \begin{enumerate}
   \item $x$ is \enf{polystable} if $G x \cap \mume \neq \vacuo$.
   \item $x$ is \enf{stable} if it is polystable and $\lieg_x$ is
     conjugate to a subalgebra of $\liek$.
   \item $x$ is \enf{semi--stable} if
     $\overline{G x} \cap \mume \neq \vacuo$.
   \item $x$ is \enf{unstable} if it is not semi--stable.
   \end{enumerate}
 \end{defin}
\begin{remark}
   The four conditions above are $G$-invariant in the sense that if a
   point $x$ satisfies one of them, then every point in the orbit of
   $x$ satisfy the same condition. This follows directly from the definition
   for polystability, semi--stability and unstability, while for 
   stability 
    it is enough to recall that
   $\lieg_{g x} = \Ad(g) (\lieg_x)$.
 \end{remark}

The following result establishes a relation between the Kempf-Ness function and polystable points.
\begin{prop}\label{critical-point}
  Let $x\in \misu$. The following conditions are equivalent:
  \begin{enumerate}
  \item $g\in G$ is a critical point of $\Psi(x, \cd)$;
  \item $\fun(gx) =0$;
  \item $g\meno K$ is a critical point of $\psi_x$.
  \end{enumerate}
\end{prop}
\begin{proof}
  Let $\x \in \liep$. Using the cocycle condition ($P4$), one gets:
  \begin{gather*}
    \Psi(x, \exp(t\x )g ) = \Psi (x,g)+ \Psi (gx,\exp(t\x ) ).
  \end{gather*}
  Therefore,
  \begin{gather}
    \label{derivata-ovunque}
    \deze \Psi(x, \exp(t\x ) g) = \deze \Psi(gx, \exp(t\x )) = \sx
    \fun (gx), \x \xs.
  \end{gather}
  Since for any $k\in K$, $\Psi(x,kg)=\Psi(x,g)$, then $\fun (gx)=0$
  if and only if $g$ is a critical point of $\Psi (x,\cdot)$ if and
  only if $g\meno K$ is a critical point of $\psi_x$.
\end{proof}
\begin{prop}\label{compatible}
   If $\fun (x)=0$, then $G_x$ is compatible.
 \end{prop}
 \begin{proof}
   Let $g\in G_x$. Then $g=k \exp (\x)$ for some $k\in K$ and
   $\x\in \liep$. By Proposition \ref{equivarianza}, we have
   $\fun (\exp (\x) x)=0$. Let $f(t):=\fun^{\x} (\exp (t\x)x)$. Then
   $f(0)=f(1)=0$ and
   \begin{gather*}
     \desudt f(t)=\desudt \fun^{\x} (\exp
     (t\x)x)=\frac{\mathrm{d^2}}{\mathrm{dt}^2 } \Psi(x,\exp(t\x))
     \geq 0.
   \end{gather*}
   Therefore
   $\frac{\mathrm{d^2}}{\mathrm{dt}^2 } \Psi(x,\exp(t\x))=0$ for
   $0\leq t \leq 1$.  It follows from ($P3$) that $\exp(t \x)x=x$ for any $t\in \R$ and thus
   $G_x$ is compatible.
 \end{proof}
Next we give a numerical criteria for an element $x\in \spaz$ to be stable. We begin with the following lemma.
\begin{lemma}\label{intersezione-nulla}
If $\lia \subset \lieg$ is a subalgebra which is conjugate to a
    subalgebra of $\liek$, then $\lia \cap \liep =\{0\}$.
  \end{lemma}
  \begin{proof}
    It is enough to show that $\Ad (g)( \liek) \cap  \liep =\{0\}$
    for any $g\in G$.
  Let $X\in \Ad (g)( \liek) \cap  \liep$. By the Cartan decomposition $G=K\exp (\liep)$, it follows $\Gamma =\exp(\R X)$ is a closed abelian subgroup of $G$ isomorphic to $\R$. On the other hand $X=\Ad(g)(Y)$ for some $Y\in \liek$ which implies   $\Gamma=\Ad(g)(\exp(\R Y))$ is a torus. Hence   $X=0$.
\end{proof}
Consider the function: 
\begin{gather*}
  \La : \spaz \times \liep \ra [-\infty, +\infty],\\
  \La(x, \xi) : = \lim_{t\to +\infty} \desudt \Psi(x,\exp(t\xi))= \lim _{t\to +\infty} \desudt \psi_x (-t\xi K  ).
\end{gather*}
The following Lemma is proven in \cite[Lemma 2.10]{teleman-symplectic-stability}.
\begin{lemma}\label{linear-proper}
Let $V$ be a subspace of $\liep$. For a point $x\in \spaz$ the following conditions are equivalent:
\begin{enumerate}
\item The map 	$\Psi(x, \exp (\xi))$ is linearly proper on $V$, i.e. there exist positive constants
$C_1$ and $C_2$ such that:
\[
||\xi ||^2 \leq C_1 \Psi(x,\exp(\xi))+C_2, \quad \forall\, \xi \in V.
\]
\item $\Lambda(x,\xi)>0$ for any $\xi\in V\setminus\{0\}$.
\end{enumerate}
\end{lemma}
\begin{teo}\label{stabile}
Let $x\in \spaz$.
Then $x$ is stable if and only if $\La(x,\xi) >0 $ for any $\xi \in \liep \setminus\{0\}$.
\end{teo}
\begin{proof}
 Let first $x\in \spaz$ be stable. Then $\fun(g x) =0$ for some $g\in G$ and by Proposition \ref{critical-point}, $g$ is a critical point of
  $\Psi(x, \cd)$. Set $y=g x$. We start by proving
  $\La (y, \xi) >0$ for any $\xi \in \liep\setminus\{0\}$.  By $(P3)$ the function
  $f(t)=\Psi(y, \exp(t\xi))$ is a convex function. Hence:
  \begin{gather*}
    \La(y,\xi) \geq f'(0)=\desudtzero \Psi(y, \exp( t\xi ) )= \sx \fun
    (y), \xi \xs=0.
  \end{gather*}
  Assume $\La (y, \xi )=0$.  By assumption $f$ is a convex function satisfying $\lim_{t\to +\infty} f'(t)= 0$ and $f'(0)=0$. Hence $f'(t)=0$ for $t\geq 0$ and so
  $\frac{\mathrm{d^2}}{\mathrm{dt}^2 } \Psi(x,\exp(t \xi ))=0$ for any
  $t\geq 0$. By ($P3$) it follows that
  $\exp(\R \xi ) \subset G_y$, so $\xi \in \lieg_y\cap \liep$.  Since
  $x$ is stable, $\mathfrak g_y = \Ad(g) (\lieg_x)$ is conjugate to a
  subalgebra of $\liek$, thus Lemma \ref {intersezione-nulla} implies
  that $\xi =0$.

  By Lemma \ref{linear-proper} the function $\Psi(y,\cdot)$ is linearly proper on $\liep$. By the cocycle condition we have
  \[
  \Psi(x,\exp(\xi))=\Psi(g^{-1}y,\exp(\xi))=\Psi(y,\exp(\xi)g^{-1})-\Psi(y,g^{-1}).
  \]
Write $\exp(\xi)g^{-1}=k(\xi )\exp(\theta(\xi))$. Then $\Psi(x,\exp(\xi))=\Psi(y,\exp (\theta (\xi)))-\Psi(y,g^{-1})$.
Using the same arguments in  \cite{mundet-Crelles}, we get  an estimate of the form
  \[
  || \xi ||^2 \leq A_1 || \theta(\xi) ||^2 + A_2,
  \]
  where $A_1$ and $A_2$ are positive constants. Therefore the linearly properness of $\Psi(y,\cdot)$ on $\liep$ implies the linearly properness of $\Psi(x,\cdot)$ on $\liep$. Hence, by Lemma \ref{linear-proper},
  $\Lambda(x,\xi)>0$ for any $\xi\in \liep\setminus\{0\}$.

Assume now that $\Lambda (x,\xi) > 0$ for any $\xi\in \liep\setminus\{0\}$. Then $\Lambda(x,\cdot)$ restricted on the unit sphere $S(\liep)$ of $\liep$ has  a minimum $C>0$.

Let $\xi \in S(\liep)$ and let $f(t)= \Psi(x,\exp(t \xi))$. The function $f$  is a convex function and  $\lim_{t\to + \infty} f'(t)\geq C$, respectively $\lim_{t\to - \infty} f'(t)\leq -C$. Hence $f$ has a global minimum and
$\lim_{t\to + \infty} f(t)=+\infty$. Thus, for any $M>0$, there exists  $t(\xi )>0$ such that  $f(t)=\Psi(x,\exp(t \xi))> M$ for any $t\geq t(\xi)$.

We claim that there exists $\gamma_o>0$ such that $\Psi(x,\exp(\xi))> \frac{M}{2}$ for $\xi \in \liep$ with $||\xi|| \geq \gamma_o$. Indeed, otherwise there exist sequences $\xi_n \in S(\liep)$ and $t_n \in \R$  with $t_n \mapsto +\infty$ such that $\Psi(x,\exp(t_n \xi_n )) \leq \frac{M}{2}$. We may assume $\xi_n \mapsto \xi_o$. Since
$\Psi(x,\exp( t\xi_o ))\geq M$ for $t>t(\xi_o)$ and keeping in mind that the function
\[
\R \times S(\liep) \lra \R, \qquad (t,\xi)\to \Psi(x,\exp(t\xi)),
\]
is continuous, there exists a neighborhood $U$ of $\xi_o$ in $S(\liep)$ and a neighborhood $(t(\xi_o)-\epsilon,t(\xi_o)+\epsilon)$ of $t(\xi_o)$ in $\R$, such that $\Psi (x, \exp (t \xi))>\frac{M}{2}$ for any $t\in (t(\xi_o )-\epsilon,t( \xi_o )+\epsilon)$ and for any $\xi\in U$. Now,  there exists $\tilde n\in \mathbb N$ such that $\xi_n \in U$ and $t_n >t( \xi_o )$ for $n\geq \tilde n$. Since the function $t\mapsto \Psi(x,\exp(t\xi))$ increases, it means $\Psi(x,\exp (t_n \xi_n ))>\frac{M}{2}$ for $n\geq \tilde n$ which is a contradiction.
Now, keeping in mind that $\psi_x \circ \exp (\xi)=\Psi (x,\exp(-\xi))$, we have proved that the function $\psi_x \circ \exp$ has a minimum and so a critical point. Since $\exp:\liep\lra G/K$ is a diffeomorphism, it follows that $\psi_x$ has a critical point. By Proposition \ref{critical-point} the point $x$ is polystable. Let $g\in G$ such that $\fun (g x)=0$. Set $y=g x$. Since
\begin{gather*}
 \La(y,\xi) \geq \desudtzero \Psi(y, \exp( t\xi ) )= \sx \fun (y), \xi \xs=0,
\end{gather*}
by the same arguments used before, we have $\La (y, \xi)>0$ for any $\xi\in \liep\setminus\{0\}$. To conclude the proof we prove  $\lieg_{y} \cap \liep=\{0\}$.

Let $\xi \in \lieg_{y} \cap \liep$. By Remark \ref{stabilizzatore-lineare} the function
$
t\mapsto \Psi( y,\exp(t\xi))
$
is linear. Since both $\La(y,\xi)$ and $\La(y,-\xi)$ are positive it follows
\[
\lim_{t\mapsto +\infty} \desudt \Psi(y,\exp(t\xi))=a\geq 0, \  \lim_{t\mapsto +\infty} \desudt \Psi(y,\exp(-t\xi))=-a\geq 0.
\]
This implies $a=0$, $\Lambda(y,\xi)=0$ and so $\xi=0$. By Proposition \ref{compatible}, $\lieg_{y}$ is a compatible subalgebra of $\lieg$ with $\lieg_{y}\cap \liep=\{0\}$. Hence $\mathrm{Ad}(g)(\lieg_x) =\lieg_{y}\subset \liek$ proving $x$ is stable.
\end{proof}
\begin{remark}
One may prove that the condition $\La(x,\xi)>0$ for any $\xi \in \liep\setminus\{0\}$ is equivalent to $\psi_x$ being an exhaustion.
\end{remark}

\begin{cor}
  \label{stabcomp}
  If $x\in \spaz$ is stable, then $G_x$ is compact.
\end{cor}
\begin{proof}
  Let $g\in G$ be such that $\fun(g x)=0$ and set $y=g x$. By
  Proposition \ref{compatible} the stabilizer of $y$, i.e. $G_y$, is compatible and so has only
  finitely many connected components. Moreover $G_y^0$ is compact
  since $\lieg_y \subset \liek$. It follows that $G_y$ and
  $G_x = g\meno G_y g $ are both compact.
\end{proof}
If $\spaz'$ is a $G$-invariant subspace of $ \spaz $, the
restriction of $\Psi $ to $G\times \spaz'$ is a Kempf-Ness function
for $(\spaz', G, K)$. The functions $\Lambda$ and $\fun$ for
$(\spaz', G, K)$ are simply the restrictions of those for $\spaz$.
If $G'\subset G$ is a compatible subgroup of $G$, i.e., $G'=K'\exp(\liep')$,
then $K'\subset K$, $\liep' \subset \liep$ and $X':= G'/K' \hookrightarrow X$ is a totally geodesic inclusion.  If $\Psi$ is a Kempf-Ness
  function for $(G,K, \spaz)$, then
  $ \Psi^{K'}:=\Psi\restr{\spaz\times G'}$ is a Kempf-Ness function
  for $(G', K', \spaz)$.  The related functions are
  \begin{gather}\label{restrizione}
    \fun^{K'}: \spaz \ra {\liep}^* , \qquad \fun^{K'}(x):=
    \fun(x)\restr
    {\liep'},\\
    \psi^{K'}_x := \psi _x \restr{X'}, \qquad \Lambda^{K'} = \Lambda
    \restr{\spaz\times \liep'}.
  \end{gather}
A subalgebra contained in $\liep$ must be abelian since $[\liep,\liep]\subset \liek$.
The following Corollary is analogous to the stability part in the
Hilbert-Mumford criterion.
\begin{cor}\label{Hilbert-Mumford}
A point $x \in \spaz$ is $G$-stable if and only if it is
$A$-stable for any  abelian group $A=\exp (\lia)$, where $\lia$ is a subalgebra of $\lieg$ contained in $\liep$.
\end{cor}
\begin{proof}
By Theorem \ref{stabile} it is enough to prove that we have $\Lambda(x,\xi) >0$ for any $\xi\in \liep\setminus\{0\}$ if and only if
for any  abelian group $A=\exp (\lia)$, where $\lia$ is a subalgebra of $\lieg$ contained in $\liep$ we have $\Lambda^A (x,\xi)>0$ for any $\xi\in \lia\setminus\{0\}$. The necessary condition is trivial, being $\Lambda^A (x,\xi)$ the restriction of $\Lambda (x,\xi)$ to $\lia$. For the sufficient, observe that for any $\xi\in S(\liep)$ we can set $\lia=\R \xi$ and conclusion follows since with this choice we have $\Lambda (x, \xi ) = \Lambda^A (x, \xi ) $.
\end{proof}
We conclude this section with the following interesting result.
\begin{cor}\label{stable-open-abstract-setting}
The function $\Lambda : \spaz \times S(\liep) \lra \R $ is lower semincontinuos  and  the set of stable points is open in $\spaz$.
\end{cor}
\begin{proof}
The proof of \cite[Lemma 3.9]{bgs} works also for $\Lambda$ proving it is lower semicontinuos.  The openness of the stable points can be proved  as in \cite[Corollary 3.10]{bgs}.
\end{proof}
\section{Polystability and semi-stability}
\label{sec:polyst-semi-stab}
The aim of this section is to characterize polystability and semi-stability of $x \in \spaz$ in terms of the maximal weight
$\la_x$. 
Throughout this section we assume that the Kempf-Ness function of $(\spaz,G,K)$ satisfies not only ($P1$)--($P5$) but also $(P6)$. Further, for semi-stability we also assume that $\spaz$ is compact. This will be enough for the case of measures on a compact manifold.

Let us denote by $\spaz^{ps}$ the set of polystable points, i.e. according to Definition \ref{stabilita}:
$$
\spaz^{ps}=\{x\in \spaz:\, G x \cap \fun^{-1}(0) \neq \vacuo\}.
$$
It follows by an easy argument that if $x\in \spaz$ is polystable then $G x\cap \fun^{-1}(0)$ contains exactly one $K$-orbit. 
Indeed, let $y\in G x$ be such that $\fun (y)=0$. We shall prove that $K y=G y \cap \fun^{-1}(0)$. Assume that $g y\in \fun^{-1}(0)$. Set
$g=k\exp(v)$. By the $K$-equivariance of $\fun$ it follows $\fun (\exp(v) y)=0$. As in the proof of Proposition \ref{compatible}, we get $\R v\in \mathfrak g_y$ and so $G y \cap \fun^{-1} (0)=K y$.
Hence we have proven the following result.
\begin{prop}
The inclusion $\fun^{-1}(0) \hookrightarrow \spaz^{ps}$ induces a bijection
\[
\fun^{-1}(0)/K \lra \spaz^{ps} /G.
\]
\end{prop}
In this section we give a numerical criteria for an element $x\in \spaz$ to be a polystable point.
Let $x\in \spaz$. We define $Z(x)=\{p\in \tits:\, \la_x (p)=0 \}.$
We start with the following result.
\begin{prop}\label{polistabile-tits}
Let $x\in \fun^{-1} (0)$. Then $\la_x \geq 0$, $\lieg_x = \liek_x \oplus  \lieq \subset \liek\oplus\liep$ is compatible
and $Z(x) = \e(S(\lieq))=\partial_\infty G_x /K_x$.
  \end{prop}
  \begin{proof}
By Proposition \ref{compatible} the stabilizer $G_x$ is compatible. Hence $\lieg_x=\liek_x\oplus \lieq$ with $\liek_x \subset \liek$ and $\lieq\subset \liep$.
Further, observe that for $\xi \in \liep$, since $\Psi(x, \exp(t\xi))$ is a convex function, we get:
\[
\Lambda(x,\xi)\geq \desudtzero \Psi(y, \exp( t\xi ) )= \sx \fun
    (y), \xi \xs=0.
\]
To conclude, we shall prove that $v\in S(\lieq)$ if and only if $\la_x (\e(-v))=0$.
Let first $\x \in S(\lieq)$. By Remark \ref{stabilizzatore-lineare} the function:
\[
f:\R \lra \R, \qquad t \mapsto \Psi(x,\exp(tv)),
\]
is linear. Since $\la_x \geq 0$,  we have $\lim_{t\rightarrow +\infty} f'(t)=a   \geq 0$ and $\lim_{t\rightarrow +\infty} f'(-t)=-a\geq 0$. Thus, $f(t)=\Psi(x,\exp(tv))=0$ and condition ($P3$) implies
$\la_x (\e (-v))=0$.

Vice-versa, assume $\la_x ( \e (-v))=0$ and consider again the function $f(t)=\Psi(x,\exp (tv))$. {Observe that }$f$ is convex and by assumptions  $\lim_{t\rightarrow + \infty} f'(t)= 0$ and $f'(0)=\desudtzero \Psi(x,\exp(tv))=\langle \fun (x),v \rangle=0$. Hence $f'(t)=0$ for $t\geq 0$. Therefore
$f''(t)=\frac{\mathrm{d^2}}{\mathrm{dt}^2 }\vert_{t=0} \Psi(x,\exp(tv))=0$.  By property $(P3)$ we get $\R v \in \lieq$ concluding the proof.
\end{proof}
Note that the inclusion $G_x/K_x \hookrightarrow X$ is totally geodesic. We claim that a vice-versa  of Proposition \ref{polistabile-tits} holds as well. We start with the following Lemma.
\begin{lemma}\label{lemma-preliminare}
Let $x\in \spaz$. Assume $\la_x\geq 0$ and $Z(x)=\partial_\infty X'$, where $X'$ is a totally geodesic submanifold of $X$. Then, there exists $g\in G$ such that setting $y=g x$ we have $Z(y)=\partial_\infty G'/K'$, where $G'$ is compatible, $G' \cap K=K'$ and $G'\subset G_y$.
\end{lemma}
\begin{proof}
Assume first $o=[K]\in X'$. We shall prove that the statement holds for $g=e$. Since $X'$ is a totally geodesic submanifold of $X$ there exists a subspace $\lieq \subset \liep${, called  \emph{Lie triple system} of $\liep$, } such that $X'=\exp(\lieq)$ and $[[\lieq,\lieq ],\lieq]\subset \lieq$ (see e.g. \cite{helgason}).
We claim $\lieq\subset \lieg_x$. Indeed, let $v\in S(\lieq))$. Since $\la_x (e(-v))=\la_x(e(v))=0$, the convex function $f(t)=\Psi(x,\exp(tv))$ satisfies $\lim_{t\to \pm \infty} f'(t)=0$. Hence $f'$ is constant and so
\[
f''(0)=\frac{d^2}{dt} \bigg \vert_{t=0} \Psi(x,\exp(tv))=0.
\]
By properties $(P3)$ we have $v\in \lieg_x$.
Let
$\lieg'=[\lieq,\lieq]\oplus \lieq$. Observe that  $\lieg'$ is a subalgebra of $\lieg$  due to the fact that $\lieq$ is a Lie triple system of $\liep$ (see e.g. \cite{helgason}{)}. Let $G'$ denote the connected subgroup of $G$ with lie algebra $\lieg'$. Hence $G'=(G'\cap K)\exp (\lieq)$ and $G'\subset G_x$. Therefore $G'=\overline{G''}=(\overline{G''} \cap K)\exp(\lieq)$ is compatible, $G'\subset G_x$ and if we denote by $K'=G'\cap K$ we have $\partial_\infty X'=\partial_\infty G'/K'$.

In general, for any $g\in G$ we can consider the totally geodesic submanifold defined by $X''=gX'$. Since by (\ref{maximal-weight-Ginvariant}) it follows $Z(gx)=g(Z(x))$, we have:
$$
Z(g x)=g\partial_\infty X'=\partial_\infty X'',
$$
and we are done.

\end{proof}
\begin{teo}\label{polystable-condition}
An element $x\in \spaz$ is a polystable point if and only if $\la_x \geq 0$ and $Z(x) = \partial_\infty X'$ for some totally geodesic submanifold $X'\subset X=G/K$.
\end{teo}
\begin{proof}
One direction is proved in Proposition \ref{polistabile-tits}. Assume $\la_x \geq 0$ and $Z(x) = \partial_\infty X'$ for some totally geodesic submanifold $X'\subset X=G/K$.
By the above lemma and property \eqref{maximal-weight-Ginvariant}, we may assume $Z(x)=\partial_\infty G'/K'$ where $G'=K'\exp(\lieq)\subset G_x$, $\lieg'=\liek'\oplus \lieq$ with $\liek' \subset \liek$ and $\lieq\subset \liep$, $G'=K'\exp (\lieq)$ and
$Z(x)=\e (S(\lieq))$.
Write $\liep=\lieq \oplus \lieq^\perp$. By a Mostow decomposition, see \cite[Th. 9.3 p. 211]{heinzner-schwarz-Cartan}, any $g\in G$ can be written as $g=k\exp(\theta)h$, where $k\in K$, $h\in G'$ and $\theta \in \lieq^\perp$. Therefore by the $K$-invariants and the cocycle condition of $\Psi$, keeping in mind that $G'\subset G_x$, we get:
\[
\Psi(x,g)=\Psi(x,k\exp(\theta)h)=\Psi(x,\exp(\theta))+\Psi(x,h).
\]
We claim that $\Psi(x,h)=0$. Indeed, $h=k \exp(v)$ with $k\in K'$ and $v\in \lieq$. Hence  $\Psi(x,h)=\Psi(x,\exp(v))$. As in the above lemma, we consider the function
$
f(t)=\Psi(x,\exp(tv))
$
which is linear due to Remark \ref{stabilizzatore-lineare}. Since $\la_x(e(\pm v))=0$, we have $\lim_{t\mapsto \pm \infty} f'(t)=0$, which implies $f\equiv 0$ and thus $\Psi(x,h)=0$. Hence
$
\Psi(x,g)=\Psi(x,\exp(\theta)).
$
Since $\Lambda(x,\cdot) >0$ on $\lieq^\perp\setminus \{0\}$, by Lemma \ref{linear-proper}  there exist positive constants $C_1$ and $C_2$ such that
\[
|| \theta ||^2 \leq C_1\Psi(x,\exp(\theta))+C_2.
\]
This means $\Psi(x,\cdot)_{|_{\lieq^\perp}}$ is an exhaustion and so it has a minimum. Since $\Psi(x,g)=\Psi(x, \exp(\theta))$ with $\theta\in \lieq^\perp$, this means that $\Psi(x,\cdot)$ has a minimum and thus a critical point. By Proposition \ref{critical-point} the point $x$ is polystable.
\end{proof}
\begin{cor}\label{cor-polystable}
Let $x\in \spaz$ be a polystable point. Then there exist $g\in G$ and an abelian subalgebra $\lia \subset \liep$ such that $\lia\subset \lieg_{gx}$ and $gx$ is $G^{\lia}$ polystable, where $G^{\lia}=\{g\in G:\, \mathrm{Ag}(g)(\xi)=\xi $ for all $\xi \in \lia\}$. Moreover
if we denote by $G_{ss}'$ the semisimple part of $G^\lia$, then $gx$ is stable with respect to $G_{ss}'$.
\end{cor}
\begin{proof}
Let $x\in \spaz$ be a polystable point and let $g\in G$ be such that $\fun (g x)=0$. Set $y=gx$. By Proposition  \ref{polistabile-tits}, $G_{y}$ is compatible, and thus $\lieg_y=\liek_y \oplus \liep_y\subset \liek\oplus \liep$, and $Z(y)=\e(S(\liep_y))$. Let $\lia$ be a maximal abelian subalgebra  of $\liep_{y}$ and let $G^{\lia}$ be the centralizer of $\lia$ in $G$. $G^\lia$ is a compatible subgroup of $G$ (see \cite{knapp-beyond}) and by \eqref{restrizione} it follows $\fun' (y)=0$  and so $y$ is polystable with respect to $G^{\lia}$.
Let $G_{ss}'$ be the semisimple part of $G^{\lia}$. By \ref{restrizione} it follows that $y$ is $G_{ss}'$ polystable and so $(\lieg_{ss}')_y$ is compatible. We claim $(\lieg_{ss}')_y \cap \liep=\{0\}$. Indeed, if $v\in (\lieg_{ss}')_y\cap \liep$, then $v\in \liep_y$  and $[v,\lia]=0$. Since $v\notin \lia$ and $\lia$ is a maximal abelian subalgebra of $\liep_y$ we get a contradiction. Since $(\lieg_{ss}')_y$ is compatible it follows $(\lieg_{ss}')_y \subset \liek$ and so $(G_{ss}')_y$ is compact. Therefore $y$ is $G_{ss}'$ stable concluding the proof.
\end{proof}

The following theorem, in analogy with \cite[Th. 4.17]{bgs}, gives a numerical criteria for semi-stable points in terms of maximal weights. The proof is the same of the proof of \cite[Theorem 4.3]{teleman-symplectic-stability} and thus it follows by \cite[Lemma 3.4]{kapovich-leeb-millson-convex-JDG} due to Kapovich, Leeb and Millson.

%
\begin{teo}\label{semi-stable-abstract}
  If $\spaz $ is compact, then a point $x \in \spaz$ is semi-stable if
  and only if $\la_x \geq 0$.
\end{teo}
The following result is a
Hilbert-Mumford criterion for semi-stability. The proof is totally similar to Corollary \ref{Hilbert-Mumford}'s one.
\begin{cor}\label{Hilbert-Mumfords}
A point $x \in \spaz$ is $G$ semi-stable if and only if it is
$A$ semi-stable for any  abelian group $A=\exp (\lia)$, where $\lia$ is a subalgebra of $\lieg$ contained in $\liep$.
\end{cor}
We conclude this section with the following {corollaries}.
\begin{cor}
Let $x\in \spaz$ be a semi-stable point. Then either $x$ is stable or  $\overline{G x}\,\cap\, \fun^{-1}(0)\subset \spaz^{ps}$.
\end{cor}
\begin{proof}
Let $x\in \spaz$ be a semi-stable point which is not stable. Setting $\spaz'=\overline{G x}$, the
restriction of $\Psi $ to $G\times \spaz'$ is a Kempf-Ness function
for $(\spaz', G, K)$ and the functions $\Lambda$ and $\fun$ for
$(\spaz', G, K)$ are simply the restrictions of those for $\spaz$. By Corollary \ref{stable-open-abstract-setting} the set of stable points of $\spaz'$ is open. By definition the set of stable points  is $G$-invariant. Hence if a point $z\in\spaz'$
were stable, then $x$ would also be stable contradicting our assumption.
\end{proof}
\begin{cor}
If $x\in \spaz$ is semistable then any $y\in \overline{G x}$ also is.
\end{cor}
\begin{proof}
Let $g_{\alpha} \in G$ be a net such that $g_{\alpha}  x \rightarrow y$ and let $v\in \tits$. By the $G$-equivariance of the maximal weight (\ref{def-la}) and the semicontinuity of $\la$, we get:
\[
\lambda_y (v)=\lambda_{g_{\alpha}^{-1}y}(g_{\alpha} v)\geq \liminf_{\alpha }\lambda_{g_{\alpha}^{-1} x} (g_{\alpha} v)\geq 0,
\]
concluding the proof.
\end{proof}
\section{The integral of the gradient map}\label{gradient-map}
Let $U$ be a compact connected Lie group and {denote by $\liu$ its Lie algebra and by} $U^\C$ its
complexification.  Let $(Z,\omega)$ be a K\"ahler manifold on which
$U^\C$ acts holomorphically. Assume that $U$ acts in a Hamiltonian
fashion with momentum map $\mu:Z \lra\liu^*$.  Consider a closed connected subgroup $G$
of $U^\C$ compatible with respect
to the Cartan decomposition of $U^\C${, i.e. $G=K\exp (\liep)$, for $K=U\cap
G$ and $\liep=\lieg \cap i\liu$
(see \cite{heinzner-schwarz-stoetzel,heinzner-stoetzel-global})}
\cite{heinzner-schwarz-stoetzel,heinzner-stoetzel-global}.  The
inclusion $i \liep \hookrightarrow \liu$ induces by restriction a
$K$-equivariant map $\mu_{i \liep}:Z \lra (i \liep)^*$.
There is a
$\mathrm{Ad}\, (U^\C)$--invariant and non-degenerate bilinear {form} $B:\liu^\C \times \liu^\C \lra \R$ which is
positive definite on $i\liu$, negative definite on $\liu$ and such
that $B(\liu,i\liu)=0$ (see \cite[p. 585]{biliotti-ghigi-heinzner-2}). 
Therefore $B$ is
$\mathrm{Ad}\, (U^\C)$--invariant, non-degenerate and its restriction
to $\lieg$ satisfies the following conditions: $B$ is $\mathrm{Ad}(G)$--invariant, $B(\liek,\liep)=0$, B restricted to $\liek$ is negative definite and $B$ restricted to $\liep$ is positive definite. Using $\scalo$, we identify $\liu
\cong \liu^*$.   For $z \in Z$, let $\mup (z) \in \liep$ denote $-i$ times the
component of $\mu(z)$ in the direction of $i\liep$.  In other words we
require that $\sx \mup (z) , \beta \xs = - \sx \mu(z) , i\beta\xs$,
for any $\beta \in \liep$. Then, we view $\mu_{ \liep}$ as a map:
\begin{gather*}
  \mu_\liep : Z \ra \liep ,
\end{gather*}
which is called the $G$-\emph{gradient map} or \emph{restricted
  momentum map} associated to $\mu$.  We also set:
\begin{gather*}
  \mupb:= \sx \mup, \beta \xs = \mup^{-i\beta}.
\end{gather*}
By definition, it follows that $\mathrm{grad}\mupb =\beta_Z$, where $\beta_Z (x)=\desudtzero \exp(t\beta) x$.

Throughout this section we fix
a $G$-stable subset $M \subset Z$ and we consider the
gradient map $\mup:M \lra \liep$ restricted on $M$. Further, we denote by $\beta_M=\desudtzero \exp(t\beta) x$. Observe that if $M$ is a manifold, then $\beta_M$  is the gradient of $\mup^\beta$ restricted to $M$ with respect to the induced Riemannian structure on $M$.

As Mundet pointed out in \cite{mundet-cont}, the existence of the Kempf-Ness function for an action of a complex reductive group on a K\"ahler manifold given in \cite{mundet-Crelles} also holds for the setting introduced in \cite{heinzner-schwarz-Cartan,heinzner-schwarz-stoetzel,heinzner-schutzdeller}.
\begin{teo}
There exists a Kempf-Ness function for $(M, G, K)$ satisfying the conditions $(P1)-(P5)$. Furthermore, if $M$ is a $G$-stable compact submanifold of $Z$, then $(P6)$ holds as well.
\end{teo}
\begin{proof}
{Fix $x\in M$. Let $\pi_\liep: \lieg \lra \liep$ be the linear projection induced by the decomposition $\lieg=\liek\oplus \liep$ and identify $T_eG$ with $\lieg$ in the usual way. For $g\in G$ and $v\in T_gG$, one has $dR_{g\meno} (v) \in \lieg$. Thus, we can define a $1$-form $\sigma$ on $G$ by setting:
\begin{gather*}
  \sigma_g(v) : = \sx \mup(g x) , \im (dR_{g\meno}(v)) \xs.
\end{gather*}
Observe that $\sigma_g \in T_gG^*$ and $\sigma \in \est^1(G)$.
}
 When we need to
stress the dependence on $x$ we will write $\sigma^x$. We claim that $\sigma$ is closed.
In order to prove it, fix $g\in G$, $v, w\in T_gG$ and let $\xi , \eta \in \lieg$ be such
that $dR_g(\xi ) = v$ and $dR_g(\eta) = w$. Further, let also $X,Y\in \XX(G)$ be the
fundamental vector fields corresponding to $\xi $ and $\eta$ under the
action of \emph{left} multiplication.  In other words $X$ is the
right-invariant vector field such that $X(e) = \x $, i.e. for $h\in G$,
\begin{gather*}
  X(h) : = dR_h(\x ) = \desudtzero \exp (t\x ) h.
\end{gather*}
For a left action the map that sends a vector in $\lieg$ to its
fundamental vector field is an anti-isomorphism of Lie algebras. Thus
$[X,Y]$ is the fundamental vector field corresponding to $-[\xi ,
\eta]$.
Hence:
\begin{gather*}
  [X,Y](g) := dR_g ( - [\xi , \eta]),\\
  \sigma( [X,Y])(g ) = \sx \mup(g x), \im([ \xi, \eta ])\xs.
\end{gather*}
We can assume by linearity that $\xi , \eta \in \liek \cup  \liep$.

It is immediate from the definition that $\sigma(X) = \sigma (Y) =
\sigma ([X,Y]) \equiv 0$ if $\xi , \eta \in \liek$. Thus recalling that:
$$
 (d\sigma) _g (v, w) = X (g)  \sigma(Y) - Y(g)  \sigma(X) -
  \sigma([X,Y])(g),
$$
for $\xi , \eta \in \liek$ the claim is proven.

Assume now that $\xi  \in \liek$ and $\eta \in \liep$.  Then
$\sigma(X) \equiv 0$ and for $h\in G$,
\[
\sigma(Y) (h) = \langle \mup(h x), \eta \xs =\mup^{\eta} (h x).
\]
By the $K$-equivariance of the gradient map we have:
\[
\begin{split}
(X \sigma(Y))(g) &= \desudtzero \sigma(Y) ( \exp (t\xi ) g) =
  \desudtzero \mup^{\eta}( \exp (t\xi ) g  x)  \\
  &=\desudtzero \langle \mathrm{Ad}(\exp(t \xi) (\mup(g x)),\eta \rangle = \langle [\xi , \mup(g  x)],\eta \rangle.
 \end{split}
\]
Thus:
\[
\begin{split}
  d\sigma(v,w) &= \langle [\xi , \mup(g  x)],\eta \rangle -\sx \mup(g  x), \im([\eta,\xi ])\xs
  \\ &=   \langle [\xi , \mup(g  x)],\eta \rangle - \sx \mup(g  x), [\eta,\xi ] \xs \\
  &=  \langle [\xi , \mup (g  x)],\eta \rangle   -\sx [ \xi , \mup (g  x) ] , \eta \xs\\ & = 0.
\end{split}
\]
Finally, we consider the last possibility, $\xi, \eta \in \liep$. In this case $[\xi , \eta ] \in \liek$ and thus $\sigma([X,Y]) \equiv 0$.
On the other hand:
\[
  (X\sigma(Y))(g) = \desudtzero \sigma(Y) ( \exp (t\xi )\cdot  g) =
   ( d\mup^{\eta} )_{(g x)}(  \xi_M) = \langle \eta_M , \xi_M \rangle,
\]
which is symmetric in $\xi$ and $\eta$, implying $d\sigma (v, w) =0$ also in this case.

This shows that $\sigma$ is closed. Let $\ga \in \Omega(G,e,e)$. Then
there exists $\ga'\in \Omega (K,e,e)$ such that $\ga  \sim i \circ \ga'$,
where $i: K\hookrightarrow G$, and thus:
\begin{gather*}
  \int_\ga \sigma = \int_{\ga'} i^*\sigma.
\end{gather*}
Since $i^*\sigma =0$, it follows that $\sigma$ is exact. Let $\Psi_x \in \cinf (G)$
be the unique function such that $\Psi_x(e) = 0 $ and $d\Psi_x=
\sigma^x$. Since ${\sigma^x}_{|_{TK}}\equiv 0$, then $\Psi_x(h)=0$ for any $h\in K$. Moreover, for any $\eta  \in \liep$, we have:
  \begin{gather*}
    ( d\Psi_x )_{(e)}(\eta ) = \mup^{\eta}  (x).
  \end{gather*}
 Thus, the function:
\begin{gather*}
  \Psi: M \times G \ra \R \qquad \Psi(x,g) : = \Psi_x(g),
\end{gather*}
satisfies conditions $(P_1)$ and $(P_5)$.
%
In order to prove $(P_3)$, compute:
\[
\begin{split}
\desudt \Psi_x(\exp(t\eta )) &= (\sigma^x)_{\exp(t\eta )} \bigl (
    \desudt \exp(t\eta ) \bigr )
    \\
    & =(\sigma^x)_{\exp(t\eta )} (dR_{\exp(t\eta) } (\eta )) \\ &= \sx
    \mup (\exp(t\eta )  x), \eta  \rangle\\ &= \mup^\eta  (\exp(t\eta )  x) .
\end{split}
\]
Therefore,
  \begin{gather*}
    \frac{\mathrm{d^2}}{dt^2}  \Psi_x(\exp(t\eta )) = d\mu^{\eta} (
    \eta_M)(\exp(t\eta )  x) = ||\eta _M||^2 (\exp(t\eta ) x),
  \end{gather*}
and thus $(P_3)$ follows.

In order to prove $(P_4)$, let $g\in G$ and $x\in M$. We claim that
  $R_g^*\sigma^x = \sigma^{g x}$. Indeed if $v\in T_hG$ and $w=
  dR_{h^{-1}}(\x )$, then:
  \begin{gather*}
    \sigma^{g x}_h (v) = \sx \mu(hg  x) , \im (w )\xs, \\
    (R_g^* \sigma^x)_h(v) = (\sigma^x)_{hg} (dR_g (v)) \xs =
   \sx \mu(hgx), \im(dR_{(hg)^{-1}}dR_g (\x ))\xs =\sx \mu(hgx), \im (w ) \xs.
  \end{gather*}

  Thus the claim is proven.  Therefore $\Psi_{gx} - R_g^*\Psi_x = c$
  is a constant. Evaluating at $h=e$ we get:
  \begin{gather*}
    c = 0 - \Psi_x(g)
  \end{gather*}
and thus:
  \begin{gather*}
    \Psi_{gx} (h) + \Psi_x(g) = \Psi_x(hg),
  \end{gather*}
  as desired. Property $(P2)$ follows by the cocycle condition together with the fact that for any $x\in M$, $\Psi_x(h)=0$ for all $h\in K$.

Finally, if $M$ is a compact $G$-stable submanifold of $Z$, then the norm square of the gradient map restricted to $M$ is bounded. Hence $\psi_x$ is Lipschitz since its differential is bounded and thus $(P6)$ holds.
\end{proof}
As direct consequence of Corollary \ref{stable-open-abstract-setting} we get the following result.
\begin{teo}
Let $M\subset Z$ be a $G$-invariant subset of $Z$. Then the set of stable points for the gradient map $\mup\!:M \lra \liep$ restricted to $M$ is open. Moreover, if $G=A=\exp(\lia)$, where $\lia\subset \liep$ is an abelian subalgebra, and  $\mu_{\lia}:M \lra \lia$ is the gradient map of $A$, then for any $\beta \in \lia$, the set
$
\{p\in M:\, A p \cap \mu_{\lia}^{-1}(\beta)\neq \emptyset\ \mathrm{and}\ \lia_p=\{0\} \}
$
is open.
\end{teo}
When $M$ is a compact $G$-stable submanifold of $Z$, Theorems \ref{polystable-condition} and \ref{semi-stable-abstract} also hold for the gradient map $\mup\!:M \lra \liep$ restricted on $M$. More precisely we have:

\begin{teo}
Let $M\subset Z$ be a compact $G$-invariant submanifold of $Z$ and let $\mup:M \lra \liep$ be the gradient map restricted to $M$.
Then $x \in M$ is semi-stable if and only if $\la_x \geq 0$. Furthermore,  a point $x\in M$ is polystable if and only if $\la_x \geq 0$ and $Z(x) = \partial_\infty X'$ for some totally geodesic submanifold $X'\subset X=G/K$.
\end{teo}

\section{Measures}
\label{sez-misure}
Let $M$ be a compact Hausdorff space. Denote by $\misu(M)$ the vector space
of finite signed Borel measures on $M$.  Observe that they are automatically
Radon \cite[Thm. 7.8, p. 217]{folland-real-analysis}.  Denote by
$C(M)$ the space of real continuous functions on $M$ which is a Banach
space with the $\sup$--norm. By the Riesz Representation Theorem
(see e.g. \cite[p.223]{folland-real-analysis}) $\misu(M)$ is the
topological dual of $C(M)$.
  We endow $\misu(M)$ with the weak-$*$ topology as dual of $C(M)$ that it is
  usually called the weak topology on measures. We use
  the symbol $\nu_\alf \weak \nu$ to denote the weak convergence of
  the net $\{\nu_\alf\}$ to the measure $\nu$.  Finally, we denote by
  $\proba(M) \subset \misu(M)$ the set of Borel probability measures
  on $M$.  It is well-known that $\pb M$ is a compact convex subset of
  $\misu(M)$. Indeed the cone of positive measures is closed and
  $\proba(M)$ is the intersection of this cone with the closed affine
  hyperplane $\{\nu\in \misu(M) : \nu(M) = 1\}$. Therefore $\proba(M)$ is
  closed and  it is
  contained in the closed unit ball in $\misu(M)$, which is compact in
  the weak topology by the Banach-Alaoglu Theorem
  \cite[p. 425]{dunford-schwartz-1}.
  Since $C(M)$ is separable, the weak topology on $\pb M$ is metrizable
 (see  \cite [p. 426]{dunford-schwartz-1}).

If $f : M \ra N$ is a measurable map between measurable spaces and
  $\nu $ is a measure on $M$, the \emph{image measure} $f\pf\nu$ is
  the measure on $Y$ such that $f\pf\nu(A) : = \nu (f\meno ( A ))$.  Observe that it
  satisfies the \enf{change of variables formula}
  \begin{gather}
    \label{eq:pushforward}
    \int_N u (y) d( f\pf\nu)(y) = \int_M u(f(x)) d\nu(x).
  \end{gather}
If $G$ acts on $M$, then we have an action on the probability measures on $M$ as follows:
\begin{gather}
    \label{azione}
    G\times \pb M \ra \pb M, \quad (g, \nu) \mapsto g_*\nu.
  \end{gather}
Let $U$ be a compact connected Lie group and $U^\C$ its
complexification. As in section \ref{gradient-map} we assume that  $G=K\exp (\liep)$ is a compatible subgroup of $U^\C$ and $M$ is a $G$-stable compact subset of a K\"ahler manifold $(Z,\omega)$. One can prove in a totally similar way as in the proof of \cite[Lemma 5.5 p. 18]{bgs} that the action \eqref{azione} is continuous with respect to the weak topology.
\begin{lemma}\label{umpfi}
  Let $X$ be a vector field on $Z$ such that its flow $\{\phi_t\}$ preserves $M$. If
  $ \nu \in \misu(M)$ and $X$ vanishes $\nu$--almost everywhere, then
  $\phi_t\pf \nu =\nu$ for any $t$. Hence, if $v \in \lieg$ and $v_M(x) =0$ for every
  $x$ outside a set of $\nu$--measure zero, then
  $\exp(\R v) \subset G_\nu$.
\end{lemma}
\begin{proof}
  Set $N:=\{p\in M : X(p) \neq 0\}$. Then $\nu(N)=0$ and for any
  $t\in \R$ and any $x\not \in N$, $\phi_t(x)=x$. In particular both
  $N$ and $M-N$ are $\phi_t$-invariant. If $A\subset M$ is measurable,
  then
  \begin{gather*}
    \phi_{-t}(A) =\phi_{-t}((A\setminus N) \sqcup (N\cap
    A))=(A\setminus N)\sqcup \phi_{-t} (N\cap A).
  \end{gather*}
  Since $\phi_{-t} (N\cap A) \subset N$,
  $\phi_t\pf\nu(A) = \nu (\phi_{-t}( A) ) =\nu(A\setminus N)=\nu(A)$.
\end{proof}
\begin{prop} \label{P5M} Let $M$, $G$, $K$ and $ \mup$ be as in $\S$ \ref{gradient-map} and let $\Psi^M$ be the Kempf-Ness function of $(M,G,K)$. The function:
  \begin{gather}
    \label{defipsim} \PsiM : \proba(M) \times G \ra \R, \quad
    \PsiM(\nu, g) : = \int_M \Psi^M (x, g) d\nu(x),
  \end{gather}
  is a Kempf-Ness function for $(\proba(M), G, K)$ satisfying conditions $(P_1)-(P_5)$. If in addition $M$ is compact then $\Psi$ also satisfies condition $(P6)$.  Moreover, if we denote by $X=G/K$, then
  \begin{gather}
    \label{defipsima}
    \psiM_\nu : X \ra \R , \quad \psiM_\nu( gK) := \PsiM (\nu,
    g\meno) = \int_M \psi^M_x(gK) d\nu(x),
  \end{gather}
  and if $\fun$ denotes the gradient map, then:
  \begin{gather}
    \fun : \proba(M) \ra \liep^*, \quad \fun (\nu) : = \int_M \mup(x).
    d\nu(x). \label{def-momento-misure}
  \end{gather}
Finally, if $\exp(\R \beta)\subset G_\nu$, for some $\beta\in \liep$, then $\fun(\nu)\in \liep^\beta$.

\end{prop}
For a sake of completeness we sketch the proof which is totally similar to that of Proposition 5.12 in \cite{bgs}.
\begin{proof}

Since $\Psi^M$ is left-invariant with respect to $K$, the same holds for $\PsiM$.

  Fix $\nu \in \proba(M)$. By
  differentiation under the integral sign $\PsiM(\nu, \cd)$ is a
  smooth function on $G$ and for $\x \in \liep$ we have:
  \begin{gather*}
    \frac{\mathrm{d^2}}{\mathrm{dt}^2 } \PsiM(\nu,\exp(t\x )) =\int_M
    \biggl ( \frac{\mathrm{d^2}}{\mathrm{dt}^2 } \Psi^M(x,\exp(t\x ))
    \biggr ) d\nu(x) \geq 0,
  \end{gather*}
  since the integrand is non-negative by $(P3)$.  If
  $
    \frac{\mathrm{d^2}}{\mathrm{dt}^2 } {\bigg \vert_{t=0} }
    \PsiM(\nu,\exp(t\x ))=0,
  $
  then:
  \begin{gather*}
    \frac{\mathrm{d^2}}{\mathrm{dt}^2 } {\bigg \vert_{t=0} }
    \Psi^M(\nu,\exp(t\x ))=0 \quad \text{ $\nu$-almost everywhere}.
  \end{gather*}
  Again by $(P3)$ this implies that $\x _M=0$ $\nu$-almost
  everywhere.  By Lemma \ref{umpfi} it follows that
  $ \exp(\R \x ) \subset G_\nu$. We have proven that $\PsiM$ satisfies
  $(P1)-(P3)$. The cocycle condition for $\PsiM$ follows
  immediately from the cocycle condition for $\Psi^M$.
  Fix $\nu \in \proba(M)$.  It is immediate to verify that the function
  $\psiM$ associated to $\PsiM$ as in \eqref{defpsi} is the one given by \eqref{defipsima}. Therefore it is clearly
 continuous on $\proba(M)$. Finally, it is easy to check that $\PsiM_\nu$ is
  Lipschitz whenever $M$ is a compact manifold.

Let $\beta \in \liep$. Since
 $
X^\beta=\{y\in M;\, \beta_M (x)=0\},
$
is the set of fixed points
$
\{y\in M:\, \exp(t\beta)y=y,\ \textrm{for all}\ t\in \R\, \},
$
 then $X^\beta$ is $G^\beta$-stable and $\mup(X^\beta)\subset \liep^\beta$ (see \cite{heinzner-schwarz-stoetzel}). If $\exp(\R \beta)\in G_\nu$, using the same argument of the proof of Proposition \ref{mpc}, then $\nu$ is supported on $X^\beta$ and so $\fun (\nu)\in \liep^\beta$.
 \end{proof}
From now on we assume that $M$ is a compact $G$-stable submanifold of $Z$. We shall compute the maximal weight using the geometry of the gradient map. We begin recalling the following slice theorem proved in  \cite{heinzner-schwarz-Cartan,heinzner-schwarz-stoetzel}.
\begin{teo}
  [Linearization Theorem]
  \label{line}
  Let $M$, $G$, $K$ and $\mup$ be as in $\S$ \ref{gradient-map}.  If $x$
  is a fixed point of $G$, then there exist an open subset
  $S \subset T_ x M$, stable under the isotropy representation of $G$,
  an open $G$-stable neighborhood $\Omega$ of $x $ in $M$ and a
  $G$-equivariant diffeomorphism $h : S \ra \Omega$.  One can further
  require that $h(0)= x$ and $dh_0 = \Id_{T_x M}$.
\end{teo}
Fix $v\in \liep$.  The \emph{gradient flow} of a function
$f\in \cinf (M)$ is usually defined as the flow of the vector field
$-\grad f$.  Let $\{\phi_t\}$ denote the gradient flow of
$\mu^v$. Since $\grad \mu^v =\beta_M$, we have
$\phi_t(x) = \exp(tv)  x$. Then the function $\mup^v$ is a Morse-Bott function \cite{heinzner-schwarz-Cartan,heinzner-schwarz-stoetzel,heinzner-stoetzel-global}.
If we denote by  $c_1 > \cds > c_r$ the critical
  values of $\mu^v$, then the corresponding level sets of $\mup^v$, $C_i:=(\mu^v)\meno ( c_i)$ are submanifolds which are
  the components of $\Crit(\mu^v)$. By Theorem \ref{line} it follows that for any $x\in M$ the limit:
  \begin{gather*}
    \limes(x) : = \lim_{t\to -\infty} \phi_t(x) = \lim_{t\to +\infty }
    \exp(t\x ) \cd x,
  \end{gather*}
  exists.
Let us denote by $W_i$ the \emph{unstable manifold} of the critical component $C_i$
  for the gradient flow of $\mu^v$:
  \begin{gather}
    \label{def-wu}
    W_i := \{ x\in M: \limes (x) \in C_i \}.
  \end{gather}
   Then:
  \begin{gather}
    \label{scompstabile}
    M = \bigsqcup_{i=1}^r W_i,
  \end{gather}
and for any $i$ the map:
  \begin{gather*}
    \limes\restr{W_i} : W_i \ra C_i,
  \end{gather*}
  is a smooth fibration with fibres diffeomorphic to $\R^{l_i}$ where
  ${l_i}$ is the index (of negativity) of the critical submanifold
  $C_i$.
\begin{prop}\label{mpc}
Let $\nu$ be a polystable measure which is not stable. Hence there exist an abelian subalgebra $\lia\subset \lieg_\nu$ such that $\nu$ is supported on $M^{\lia}=\{x\in M:\, \xi_M (x)=0$ for any $\xi \in \lia\}$.
\end{prop}
\begin{proof}
By Proposition \ref{polistabile-tits}, Lemma \ref{lemma-preliminare} and Theorem \ref{polystable-condition}, $\lieg_\nu=\mathrm{Ad}(g)(\liek' \oplus \lieq)$, i.e., it is conjugate to a compatible subalgebra of $\lieg$ and $\partial_\infty G_\nu /K_\nu=Z(\nu)=g(e(S(\lieq)))$.

Let $\lia'\subset \lieq$ be a maximal abelian subalgebra of $\lieq$. Then $\lia=\mathrm{Ad}(g)(\lia')$ is an abelian subalgebra of $\lieg_\nu$ and $S(\lia)\subset Z(\nu)$.
Let $u\in \lia$. Then $\exp(tu) \in G_\nu$ and thus:
\[
\lim_{n\mapsto -\infty} \exp(n u)\nu =\nu.
\]
Let $A\subset M$ be a measurable subset. Then $\nu(A)=\lim_{n\mapsto -\infty} \nu(\exp (n u) (A))=\nu(\alpha (A))$, where $\alpha$ is the gradient flow of $\mup^{u}$.
Hence $\nu$ is supported on the critical submanifolds of $\mup^{u}$ for any $u \in \lia$. Hence $\nu$ is supported on $M^{\lia}$.
\end{proof}
Now, we explicitly compute the maximal weights.
\begin{teo} \label{calcolo-lambda} With the notation above we have
  \begin{gather*}
    \la_\nu(\e(-v)) = \sum_{i=1}^r c_i \nu(W_i).
  \end{gather*}
\end{teo}
We give a sketch of the proof, which follows essentially that of \cite[Th. 5.23]{bgs}.

\begin{proof}
 By definition of $\la_\nu$ and by
  differentiating under the integral sign we get
  \[
  \begin{split}
    \la_\nu (\e (-v))
    &=\lim_{t\to +\infty} \desudt \int_M \Psi^M (x, \exp(tv)) d\nu(x)\\
    &=\lim_{t\to +\infty} \int_M \Bigl ( \desudt \Psi^M (x, \exp(tv))
    \Bigr ) d\nu(x) .
  \end{split}
  \]
Applying the dominated convergence theorem, since $\desudt \bigg \vert _{t=t_o} \Psi^M (x, \exp(t\x ) ) = \mup^\x
    (\exp(t_0 \x ) \cd x)$
and  $\mup^v$ is bounded, we get
  \[
  \begin{split}
  \la_\nu (\e (-v)) &= \lim_{t\to +\infty} \int_M \mup^v (\exp
    (tv)\cd x) d\nu(x) \\
    &= \int_M \mu^v (\alfa(x)) d\nu(x)
    = \sum_{i=1}^r \int_{W_i} \mup^v (\alfa(x)) d\nu(x) .
  \end{split}
\]
Since for $x\in W_i$, $\alfa(x) \in C_i$ and so $\mup^v(\alfa(x)) =
  c_i$, we finally obtain:
  \begin{gather*}
  \la_\nu (\e (-v))=  \sum_{i=1}^r c_i \nu(W_i).
  \end{gather*}
\end{proof}
Let $E(\mup)$ denote the convex hull of $\mup(M)\subset \liep$, i.e. a $K$-invariant convex body in $\liep$.
Let $\lia\subset \liep$ be a abelian subalgebra and let $\pi:\liep \lra \lia$ be the orthogonal
projection onto $\lia$. Then $\mua=\pi \circ \mup$ is the gradient map associated to $A=\exp(\lia)$. Denote by $P=\mua(M)$. It is well-known that $P$ is a finite union of polytopes \cite{heinzner-schutzdeller}.
In the sequel we always assume that $P$ is a polytope, hypothesis which holds e.g.
if $G=U^\C$ and $M$ is a complex connected submanifold by the
Atiyah-Guillemin-Sternberg convexity theorem
\cite{atiyah-commuting,guillemin-sternberg-convexity-1} or when $M$ is an irreducible semi-algebraic subset of a Hodge manifold
$Z$ \cite{bghc,heinzner-schutzdeller,kostant-convexity}. We point out that the convex bodies $E(\mup)$ and $P$ are strongly related \cite{bgh-israel-p}.
Observe that under such hypothesis for any $v \in \liep$, a local maxima of  $\mup^v$ is a global maxima. Then the Morse-Bott decomposition of $M$ with respect to $\mup^v$, i.e., $M = \bigsqcup_{i=0}^r W_i$,
has a unique open and dense unstable manifold $W_r^u$ and the others unstable manifolds are proper submanifolds. Therefore, if $\nu$ is a smooth measure of $M$ then $W_r^u$ has full measure and so
$
\la_\nu (\e (-v))=  c_r =\mathrm{max}_{x\in M} \mup^v.
$
Summing up we have proved the following result.
\begin{cor}\label{misure-liscie-lambda}
If $\nu$ is a smooth measure on $M$, then for any $v\in \liep$:
\[
\la_\nu ( \e(-v) )=\mathrm{max}_{x\in M} \mup^v.
\]
\end{cor}
Since $\nu$ is a probability measures, it follows that $\fun (\nu) \in E(\mup)$. Indeed, $\fun (\nu)$ is the barycenter of the gradient map $\mup$ with respect to $\nu$ and so it lies in $E(\mup)$.  If $0\notin E(\mup)$, then there exists $v\in E(\mup)$ realizing the minimum distance of $E(\mup)$ to the origin. Moreover $v$ is a $K$ fixed point due to the fact that $E(\mup)$ is $K$-invariant. Hence  up to shifting the gradient map we may assume that $0\in E(\mup)$. Under this assumption we get the following result.
\begin{prop}\label{measure-semistable}
If $0\in E(\mup)$ then any smooth measure on $M$ is semi-stable.
\end{prop}
\begin{proof}
Let $v\in \liep$. By the above corollary, we have
$\la_\nu (\e(-v))=\mathrm{max}_{x\in M} \mup^v.$ Since $0\in E(\mup)$, it follows that $\la_\nu ( \e(-v) )=\mathrm{max}_{x\in M} \mup^v \geq 0$. By Theorem \ref{semi-stable-abstract} $\nu$ is semi-stable.
\end{proof}
\begin{cor}\label{stable-semi-stable-dense}
  If $0\in E(\mup) $, then the set
  $\proba_{ss} (M) :=\{\nu\in \proba (M):\, \nu$ is semi-stable$\}$
  is dense in $\pb M$. Moreover, if $0$ lies in the interior of $E(\mup)$ then  the set
  $\proba_{s} (M) :=\{\nu\in \proba (M):\, \nu\ \mathrm{is\ stable} \}$ is open and dense.
\end{cor}
\begin{proof}
By the above Proposition any smooth measure is semi-stable. Since smooth measures are
dense, then the set of semi-stable measures is dense. If $0$ belongs to the interior of the $E(\mup)$, then for any $v\in \liep$ the function $\mup^v$ change sign and so it has a strictly positive maxima.
By Corollary   \ref{misure-liscie-lambda} $\la_\nu ( \e(-v))>0$ and  by Theorem \ref{stabile} we get that it is stable. Since by Corollary \ref{stable-open-abstract-setting} the set of the stable points is also open, it means $\proba_s (M)$ is open and dense.
\end{proof}
\section{Measures on real projective spaces}\label{misure-proiettivo}

In the recent paper \cite{bgs} the authors completely describe stable, semi-stable and polystable measures on  complex projective spaces (see also \cite{donaldson-numerical-results,millson-zombro}).
Here we consider the real projective space:
$$
\mathds{P}^n(\mathds{R})=\frac{\mathds{R}^{n+1}\setminus \{0\}}{\sim}=\frac{\mathds{S}^n}{\{\pm {\rm Id}_{n+1}\}},
$$
where we denote by ${\rm Id}_{n+1}$ the identity matrix of order $n+1$. Consider on $\mathds{P}^n(\mathds{R})$ the action of ${\rm SL}(n+1,\mathds{R})$  and recall that its Lie algebra $\mathfrak{sl}(n+1)$ decomposes as $\mathfrak{sl}(n+1)=\mathfrak{k}\oplus\mathfrak{p}=\mathfrak{so}(n+1)\oplus{\rm sym}_0(n+1)$.
A gradient moment map for this action is given by:
$$
\mu_{\mathfrak p}\!:\mathds{P}^n(\mathds{R})\rightarrow \mathfrak p,\quad \mu_{\mathfrak p}([x])=\frac12\left[\frac{xx^T}{|x|^2}-\frac1{n+1}{\rm Id}_{n+1}\right].
$$
Observe that ${\rm sym}_0(n+1)$ admits the maximal abelian subalgebra $\mathfrak a$ of traceless diagonal matrices, which we identify with $\mathds{R}^n\subset \mathds{R}^{n+1}$. Given an element $\xi\in {\rm sym}_0(n+1)$, let $\lambda_1>\dots> \lambda_k$ be its eigenvalues and denote by $V_1,\dots, V_k$ the corresponding eigenspaces.  In view of the orthogonal decompositions $\mathds{R}^{n+1}=V_1\oplus\dots\oplus V_{k}$ we can write $x\in \mathds{R}^{n+1}$ as $x=x_1+\dots+x_k$ with $x_j\in V_j$, $j=1,\dots, k$. With this notation we have:
$$
\mu_{\mathfrak a}^\xi([x])=\frac12\frac{\lambda_1|x_1|^2+\dots+\lambda_{k}|x_{k}|^2}{|x_1|^2+\dots+|x_k|^2},
$$
where $\langle\cdot,\cdot\rangle$ is the dual pairing. Consider the projection $\pi\!: \mathds{R}^{n+1}\setminus\{0\}\rightarrow \mathds{P}^n(\mathds{R})$. Since $(d\pi)_x\left(\xi_{\mathds{R}^{n+1}\setminus\{0\}}(x)\right)=\xi_{\mathds{P}^n(\mathds{R})}$ and $\xi_{\mathds{R}^{n+1}\setminus\{0\}}(x)=\lambda_1x_1+\dots+\lambda_kx_k$, one has $\xi_{\mathds{P}^n(\mathds{R})}\equiv 0$ iff $\xi_{\mathds{R}^{n+1}\setminus\{0\}}(x)$ is parallel to $x$, i.e. iff $x=x_j$ for some $j=1,\dots, k$. Thus,
critical points of $\mu_{\mathfrak p}^\xi$ are given by ${\rm Crit}(\mu^\xi_{\mathfrak p})=\mathds{P}(V_1)\cup\dots\cup \mathds{P}(V_{k})$ and critical values are $c_j=\frac12\lambda_j$, $j=1,\dots, k$.

 In order to describe:
$$
W_j^\xi=\{ [x]\in \mathds{P}^n(\mathds{R}): \limes ([x]) \in C_j \},
$$
 for $j=1,\dots, n+1$, where by definition:
$$
 \limes([x]) = \lim_{t\to +\infty}  \exp(t\xi )  x,
$$
observe that:
$$
 \exp(t\xi )  x=[\exp(t\lambda_1)x_1+\dots+\exp(t\lambda_{k})x_{k}],
$$
which implies:
\begin{equation}
 \limes([x]) =\lim_{t\to +\infty}[ \exp(t\lambda_1)x_1+\dots+\exp(t\lambda_{k})x_{k}]=\begin{cases} [x_1]\  {\textrm{if}}\ \ x_1\neq 0;\\
[x_2]\  {\textrm{if}}\ \ x_1=0, x_2\neq 0;\\
 \quad\vdots\\
 [x_k] \ {\textrm{otherwise}.}
 \end{cases}\nonumber
\end{equation}
Thus, since $[x]\in W_j^\xi$ iff $\limes([x])\in \mathds{P}(V_j)$ we have:
$$
W_1^\xi=\mathds{P}^{n}(\mathds{R})\setminus \mathds{P}(V_2\oplus\dots\oplus V_k),
$$
$$
W_2^\xi=\mathds{P}(V_2\oplus\dots\oplus V_k)\setminus \mathds{P}(V_3\oplus\dots\oplus V_k),
$$
$$
\vdots
$$
$$
W_{k-1}^\xi=\mathds{P}(V_{k-1}\oplus V_k)\setminus \mathds{P}(V_k).
$$
$$
W_k^\xi=\mathds{P}(V_k).
$$

%

By Theorem \ref{calcolo-lambda} it follows:
\begin{equation}\label{lambdaxi}
\begin{split}
\lambda_\nu(e(-\xi))=&\frac12\left(\sum_{j=1}^r\lambda_j\nu (W_j^\xi)\right)\\
=&\frac12\left(\lambda_1-(\lambda_1-\lambda_2)\nu(\mathds{P}(V_2\oplus\dots\oplus V_k))-\dots-(\lambda_{k-1}-\lambda_k)\nu(\mathds{P}(V_k))\right).
\end{split}
\end{equation}

In the following two examples we develop in details the cases $n=1$ and $n=2$.

\begin{Example}\rm
Let $n=1$.
We have $\xi=(\lambda_1,-\lambda_1)$ and $\mathds{R}^2=V_1\oplus V_2$.
Denote $p_i=\mathds{P}(V_i)$ for $i=1,2$. Then, ${\rm Crit}(\mu^\xi)=\{p_1,p_2 \}$. If {we denote $x=x_1+x_2$ as before, we have:
\begin{equation}
 \limes(X) =\lim_{t\to +\infty} [\exp(t\lambda_1)x_1+\exp(t\lambda_{2})x_{2}]=\begin{cases} p_1\  {\textrm{if}}\ \ x_1\neq 0;\\
p_2 \  {\textrm{if}}\ \ x_1=0,
 \end{cases}\nonumber
\end{equation}
which implies:
$$
W_1^\xi=\mathds{P}^1(\mathds{R})\setminus p_2,\ \  W_2^\xi=p_2.
$$
It follows that:
$$
\lambda_\nu(e(-\xi))=\frac{\lambda_1}2(1-2\nu (p_2)).
$$
Thus $\nu$ is stable iff for any $p\in \mathds{P}^1(\mathds{R})$:
$$
\nu (p)<\frac12,
$$
semistable iff for any $p\in \mathds{P}^1(\mathds{R})$:
$$
\nu (p)\leq \frac12,
$$
polystable but not stable  iff $\nu$ is only supported by two points, i.e.:
$$
\nu=\frac12 \delta_{1}+\frac12 \delta_{2}.
$$
Indeed, If $\nu$ is polystable, by Corollary \ref{cor-polystable}, there exists $\xi \in \liep$ such that $\exp(t\xi)\in \mathrm{SL}(2,\R)_\nu$, $\nu$ is supported by two points $p_1$ and $p_2$ and by:
\[
0=\lambda_\nu(e(-\xi))=\frac{\lambda_1}2(1-2\nu (p_2))
\]
it follows $\nu=\frac12 \delta_{p_1}+\frac12 \delta_{p_2}$. Vice-versa, if $\nu=\frac12 \delta_{p_1}+\frac12 \delta_{p_2}$ with $p_1 \neq p_2$, then there exists $g\in \mathrm{SL}(2,\R)$ such that $gp_1=[1:0]$ and $gp_2=[0:1]$. It is easy to check that
$$
\fun (g\nu)=\frac12(\mup([1:0])-\mup([0:1])=0,
$$
proving $\nu$ is polystable.
}
\end{Example}
\begin{Example}\rm
Let $n=2$. We have three cases:
\begin{enumerate}
\item[($a$)] $\xi=(\lambda_1,\lambda_2,\lambda_3)$, with $\lambda_3=-\lambda_1-\lambda_2$, $\mathds{R}^3=V_1\oplus V_2\oplus V_3$, $\dim(V_j)=1$;
\item[($b$)] $\xi=(\lambda_1,-\frac12\lambda_1,-\frac12\lambda_1)$ and $\mathds{R}^3=V_1\oplus V_2$, where $\dim(V_1)=1$, $\dim(V_2)=2$;
\item[($c$)]  $\xi=(\lambda_1,\lambda_1,-2\lambda_1)$ and $\mathds{R}^3=V_1\oplus V_2$, where $\dim(V_1)=2$, $\dim(V_2)=1$.
\end{enumerate}
Let us deal first with the case ($a$). Denote $p_i=\mathds{P}(V_i) \subset \mathds{P}^2 (\R)$ for $i=1,2,3$ and let $\xi=(\lambda_1,\lambda_2,\lambda_3)$. Then ${\rm Crit}(\mu^\xi_{\mathfrak p})=\{p_1,p_2,p_3\}$ and:
\begin{equation}
 \limes(x) =\lim_{t\to +\infty} [\exp(t\lambda_1)x_1+\exp(t\lambda_{2})x_{2}+\exp(t\lambda_{3})x_{3}]=\begin{cases} p_1\  {\textrm{if}}\ \ x_1\neq 0;\\
p_2\  {\textrm{if}}\ \ x_1=0,\, y_2\neq 0;\\
p_3\  {\textrm{if}}\ \ [x]=p_3,
 \end{cases}\nonumber
\end{equation}
and
$$
W_1^\xi=\mathds{P}^2(\mathds{R})\setminus \mathds{P}(V_2\oplus V_3),\ \  W_2^\xi=\mathds{P}(V_2\oplus V_3)\setminus p_3,\ \ W_3^\xi=p_3.
$$
It follows that:
 \begin{equation}
\begin{split}
\lambda_\nu(e(-\xi))=&\frac{\lambda_1}2-\frac{\lambda_1-\lambda_2}2\nu (\mathds{P}(V_2\oplus V_3))-\frac{2\lambda_2+\lambda_1}2\nu (p_3)\\
=&\frac{\lambda_1}2\left(1-\left(1-\frac{\lambda_2}{\lambda_1}\right)\nu (\mathds{P}(V_2\oplus V_3))-\left(2\frac{\lambda_2}{\lambda_1}+1\right)\nu (p_3)\right).
\end{split}\nonumber
\end{equation}
Observe that from $\lambda_1>\lambda_2>-\lambda_1-\lambda_2$ we get $-1/2<{\lambda_2}/{\lambda_1}<1$.

For the case ($b$), namely for $\xi=(\lambda_1,-\frac12\lambda_1,-\frac12\lambda_1)$, we have ${\rm Crit}(\mu^\xi)=\{p_1\}\cup\mathds{P}(V_2)$,
\begin{equation}
 \limes(x) =\lim_{t\to +\infty} [\exp(t\lambda_1)x_1+\exp(t\lambda_{2})x_{2}]=\begin{cases} p_1\  {\textrm{if}}\ \ x_1\neq 0;\\
[0:y_2]\  {\textrm{if}}\ \ x_1=0,
 \end{cases}\nonumber
\end{equation}
and
$$
W_1^\xi=\mathds{P}^2(\mathds{R})\setminus\mathds{P}(V_2)=p_1,\ \  W_2^\xi=\mathds{P}(V_2)=\mathds{P}^2(\mathds{R})\setminus\{p_1\}.
$$
It follows that:
$$
\lambda_\nu(e(-\xi))=\lambda_1\left(\frac{1}4-\frac{3}4\nu(p_1)\right).
$$

Finally, when $\xi=(\lambda_1,\lambda_1,-2\lambda_1)$, ${\rm Crit}(\mu^\xi)=\mathds{P}(V_1)\cup\{p_3\}$,
\begin{equation}
 \limes(x) =\lim_{t\to +\infty} [\exp(t\lambda_1)x_1+\exp(t\lambda_{2})x_{2}]=\begin{cases} [x_1]\  {\textrm{if}}\ \ x_1\neq 0;\\
p_3\  {\textrm{if}}\ \ x_1=0,
 \end{cases}\nonumber
\end{equation}
and
$$
W_1^\xi=\mathds{P}^2(\mathds{R})\setminus\{p_3\},\ \  W_2^\xi=\mathds{P}(V_2)=\{p_3\}.
$$
It follows that:
$$
\lambda_\nu(e(-\xi))=\frac{\lambda_1}2\left(1-3\,\nu (p_3)\right).
$$

Denote by $\rm Li\subset \mathds{R}^3$ a linear subspace of $\mathds{R}^3$ of dimension $2$ and let $p\in \mathds{P}^2(\mathds{R})$.
Then, $\nu$ is stable iff for any choice of ${\rm Li}$ and $p$:
$$
\nu(\mathds{P}({\rm Li}))<\frac23,\quad \nu(p)<\frac13,
$$
$\nu$ is semistable iff for any choice of ${\rm Li}$ and $p$:
$$
\nu(\mathds{P}({\rm Li}))\leq\frac23,\quad \nu(p)\leq\frac13,
$$
and $\nu$ is polystable iff either it is stable or it is one of the following:
$$
\nu:=\frac23\delta_{\mathds{P}({\rm Li})}+\frac13\delta_p,\qquad \nu:=\frac{1}3\delta_{1}+\frac{1}3\delta_{2}+\frac{1}3\delta_{3},
$$
i.e. it is supported by some $\mathds{P}({\rm Li})$ and by a point $p$ or by three points (see the proof of Prop. \ref{projstab} below for details).
%
%
%

\end{Example}
We conclude with the following proposition which states necessary and sufficients conditions for stability and polystability in general dimension.
\begin{prop}\label{projstab}
The measure $\nu$ is stable iff for any choice of a linear subspace ${\rm Li}\subset \mathds{R}^{n+1}$:
$$
\nu(\mathds{P}({\rm Li}))<\frac{\dim({\rm Li})}{n+1},
$$
$\nu$ is semistable iff:
$$
\nu(\mathds{P}({\rm Li}))\leq\frac{\dim({\rm Li})}{n+1}.
$$
The measure  $\nu$ is polystable iff there exists a splitting $\R^{n+1}={\rm Li_1} \oplus \cdots \oplus {\rm Li_r}$ such that $\nu$ is supported on $\mathds{P}({\rm Li}_1)\cup \cdots \cup \mathds{P}({\rm Li}_r)$. Moreover
$$
\nu:=\sum_{j}^r\frac{\dim({\rm Li}_j)}{n+1}\delta_{\mathds{P}({\rm Li}_j)},
$$
where $\delta_{\mathds{P}({\rm Li}_j)}$ is a stable measure of $\mathds{P}({\rm Li}_j)$ with respect to $\rm{SL}({\rm Li}_j )$.
\end{prop}
\begin{proof}
As before, let $\xi\in \mathfrak a$, $\lambda_1>\dots>\lambda_k$ be its eigenvalues and $V_1,\dots, V_k$ be the corresponding eigenspaces, with $\sum_{j=1}^k\dim(V_j)\lambda_j=0$. From \eqref{lambdaxi} we have $\lambda_\nu(e(-\xi))>0$ iff:
\begin{equation}\label{iffstable}
\lambda_1-(\lambda_1-\lambda_2)\nu(\mathds{P}(V_2\oplus\dots\oplus V_k))-\dots-(\lambda_{k-1}-\lambda_k)\nu(\mathds{P}(V_k))>0.
\end{equation}
Assume that $\nu(\mathds{P}({\rm Li}))<\frac{\dim({\rm Li})}{n+1}$ for any linear subspace ${\rm Li}\subset \mathds{R}^{n+1}$. Then, since $\lambda_j-\lambda_{j+1}>0$:
\begin{equation}
\begin{split}
\lambda_1&-(\lambda_1-\lambda_2)\nu(\mathds{P}(V_2\oplus\dots\oplus V_k))-\dots-(\lambda_{k-1}-\lambda_k)\nu(\mathds{P}(V_k))>\\
&>\lambda_1-(\lambda_1-\lambda_2)\frac{\dim(V_2)+\dots+\dim(V_k)}{n+1}-\dots-(\lambda_{k-1}-\lambda_k)\frac{\dim({V_k})}{n+1}=0,
\end{split}\nonumber
\end{equation}
where the last equality follows by applying  $\sum_{j=1}^k\dim(V_j)\lambda_j=0$ several times.
Viceversa, let ${\rm Li}$ be a linear subspace of $\mathds{R}^{n+1}$ of dimension $0<r<n+1$ such that $\nu(\mathds{P}({\rm Li}))\geq\frac{\dim({\rm Li})}{n+1}$. Then, $\mathds{R}^{n+1}={\rm Li}\oplus {\rm Li}^\bot$, where we denote by ${\rm Li}^\bot$ the orthogonal complement of ${\rm Li}$, and we can choose $\xi$ is such a way that $\xi=(\lambda_1,\lambda_2)$, $r\lambda_1+(n+1-r)\lambda_2=0$, with corresponding eigenspaces ${\rm Li}$ and ${\rm Li}^\bot$. We can assume without loss of generality that $\lambda_1>0$. Conclusion follows since by \eqref{lambdaxi} we have:
$$
\lambda_\nu(e(-\xi))=\lambda_1-(\lambda_1-\lambda_2)\nu(\mathds{P}({\rm Li}^\bot))\leq\lambda_1-\lambda_1\frac{n+1-r}{n+1}-\lambda_1\frac{r}{n+1}=0,
$$
where we use that $r\lambda_1+(n+1-r)\lambda_2=0$.

In order to prove the polystability part, assume that $\nu$ is polystable. Then there exists $g\in {\rm SL}(n+1,\R)$ such that $\fun (g\nu)=0$. Set $\nu'=g\nu$.  By Lemma \ref{lemma-preliminare} and Proposition \ref{mpc} there exists an abelian subalgebra $\lia \subset {\rm sym_0}(n+1)$ such that $\nu'$ is supported on $\mathds{P}^n (\R)^{\lia}$. We can diagonalize simultaneously any element of $\lia$. Hence there exists an orthogonal  splitting:
\[
\R^{n+1}=V_1 \oplus \cdots \oplus V_r,
\]
such that for any $\xi \in \lia$, we have $\xi_{|_{V_i}}=\lambda_j (\xi) \mathrm{Id}_{V_j}$. Therefore
$\mathds{P}^n (\R)^{\lia}=\mathds{P}(V_1)\cup \cdots \cup \mathds{P}(V_r)$ and so $\nu'$ is supported on $\mathds{P}(V_1)\cup \cdots \cup \mathds{P}(V_r)$. This means that
$\nu'=\sum_{j=1}^r \lambda_i \delta_{\mathds{P}(V_j)}$, where $\delta_{\mathds{P}(V_j)} \in \proba (\mathds{P}(V_j))$, $\lambda_j \geq 0$ for $j=1,\ldots,r$ and $\sum_{j=1}^r \lambda_j=1$. Since $\mathrm{SL}(n+1,\R)^{\lia}=\mathrm{SL}(V_1 \oplus \cdots \oplus V_r)$ its semisimple part is given by
$\mathrm{SL} (V_1) \times \cdots \times\mathrm{SL} (V_r )$. By Corollary \ref{cor-polystable} $\nu'$ is $\mathrm{SL} (V_1) \times \cdots \times \mathrm{SL} (V_r )$ stable and so its stabilizer:
\[
(\mathrm{SL} (V_1) \times \cdots \times \mathrm{SL} (V_r ))_{\nu'}=\mathrm{SL} (V_1)_{ \delta_{\mathds{P}(V_1)}}\times \cdots \times \mathrm{SL} (V_r)_{ \delta_{\mathds{P}(V_r)}}
\]
is compact. In particular $\mathrm{SL} (V_j)_{ \delta_{\mathds{P}(V_j)}}$ is compact.
If we decompose $x=x_1+\cdots +x_r$ by means of the above splitting, we have:
\[
\begin{split}
0=\fun (\nu') &= \int_{\mathds{P}^n (\R) } \mup(x)  \mathrm{d}\nu'(x)= \int_{\mathds{P}^n (\R)} \frac{xx^T}{|| x ||^2}  \mathrm{d} \nu'(x) - \frac{1}{n+1}\mathrm{Id}_{n+1}\\
&=\sum_{j=1}^r \lambda_j \int_{\mathds{P}(V_j)} \left(\frac{x_jx_j^{T}}{|| x_j ||^2} - \frac{1}{\dim V_j}\mathrm{Id}_{V_j}\right)  \mathrm{d} \delta_{\mathds{P}(V_j)}(x_j) +\sum_{j=1}^r \lambda_j \frac{1}{\dim V_j} -\frac{1}{n+1} \mathrm{Id}_{n+1} \\
&=\sum_{j=1}^r \lambda_j \fun^j (\delta_{\mathds{P}(V_j)}) + \sum_{j=1}^r \left(\frac{\lambda_j}{\dim V_j} - \frac{1}{n+1}\right)\mathrm{Id}_{V_j}.
\end{split}
\]

In the above formula $\fun^j$ denotes the gradient map with respect to the $\mathrm{SL}(V_j)$ action on
$\proba ( \mathds{P}(V_j))$. Therefore, keeping in mind that $\sum_{j=1}^r \lambda_j \fun^j (\delta_{\mathds{P}(V_j)})$, which lies in $\mathfrak{sym_0}(n+1)$, and $\sum_{j=1}^r \frac{\lambda_j}{\dim V_j}\mathrm{Id}_{V_j} - \frac{1}{n+1} \mathrm{Id}_{V_j}$ are orthogonal in $\mathfrak{gl}(n+1,\R)$, we have $\fun^j (\delta_{\mathds{P}(V_j)})=0$, and so by the above discussion $\delta_{\mathds{P}(V_j)}$ is stable with respect to $\mathrm{SL}(V_j) $, and $\lambda_j=\frac{ \dim V_j }{n+1}$. Set ${\rm Li_j}=g^{-1} V_j$ for any $j=1,\ldots,r$. By the above discussion $\nu=\sum_{j}^r\frac{\dim({\rm Li}_j)}{n+1}\delta_{\mathds{P}({\rm Li}_j)},
$
where $\delta_{\mathds{P}({\rm Li}_j)}$ is a measure of $\mathds{P}({\rm Li}_j)$. We claim $\delta_{\mathds{P}({\rm Li}_j)}$ is a stable measure with respect to $\rm{SL}({\rm Li}_j )$. Indeed,
$\rm{SL}({\rm Li}_j )=g^{-1}\mathrm{SL}(V_j)g$ and it is easy to check that:
\[
\mathrm{Ad}(g^{-1})\circ \mup^{\mathrm{SL}(V_j)}=\mup^{\mathrm{SL}(\rm{Li}_j)} \circ g^{-1}.
\]
Similarly $\mathrm{Ad}(g^{-1}) \circ \fun=\fun'\circ g^{-1}$ and so
$\fun^{-1}(0)=g\cdot \fun'^{-1}(0)$ proving  $\delta_{\mathds{P}({\rm Li}_j)}$ is a stable measure with respect to $\rm{SL}({\rm Li}_j )$.

Vice-versa, assume $\nu=\sum_{j}^r\frac{\dim({\rm Li}_j)}{n+1}\delta_{\mathds{P}({\rm Li}_j)}
$ with respect to a splitting:
\[
\R^{n+1}={\rm Li_1} \oplus \cdots \oplus {\rm Li_r},
\]
where $\delta_{\mathds{P}({\rm Li}_j)}$ is a stable measure of $\mathds{P}({\rm Li}_j)$ with respect to $\rm{SL}({\rm Li}_j )$.
Let $g\in \mathrm{SL}(n+1,\R)$ such that if we denote by $V_j=g{\rm Li_j}$ for $j=1,\ldots,r$, then:
\[
\R^{n+1}=V_1 \oplus \cdots \oplus V_r,
\]
is an orthogonal splitting. By the above computation we get $\fun (g\nu)=\sum_{j=1}^r \frac{ \dim V_j }{n+1} \fun^j (\delta_{\mathds{P}(V_j)})$, where $\delta_{\mathds{P}(V_j)}=g\delta_{\mathds{P}({\rm Li}_j)}$ for $j=1,\ldots,r$. By the above discussion since $\delta_{\mathds{P}({\rm Li}_j)}$ is stable with respect to $\rm{SL}({\rm Li}_j )$, then $g\delta_{\mathds{P}(V_j)}$ is stable with respect to $\mathrm{SL}(V_j)$. Hence there exists $g_j \in \rm{SL}(V_j)$ such that $\fun^j (g_j  \delta_{\mathds{P}(V_j)})=0$. Let $h=g_1\times \cdots \times g_r \in \mathrm{SL} (V_1) \times \cdots \times \mathrm{SL} (V_r )\subset \mathrm{SL}(n+1,\R)$. Then
\[
\fun(hg\nu)=\sum_{j=1}^r \frac{ \dim V_j }{n+1} \fun^j (g_j\delta_{\mathds{P}(V_j)})=0,
\]
concluding the proof.
\end{proof}
\def\cprime{$'$}

\end{document}